\documentclass[hidelinks,a4paper,11pt]{article}

\pdfoutput=1
\usepackage{graphicx}
\usepackage{amssymb}
\usepackage{float}
\usepackage{anysize}
\usepackage{multirow}
\usepackage{verbatim}
\usepackage{amsthm,amsmath}
\usepackage{hyperref}		
\newcounter{dummy} 
\numberwithin{dummy}{section}
\newtheorem{thm}[dummy]{Theorem}
\theoremstyle{definition}
\newtheorem{defn}[dummy]{Definition}
\numberwithin{equation}{section}
\marginsize{3cm}{2cm}{1cm}{1cm}

\begin{document}

\title{\textbf{The Martingale Representation Theorem and Clark-Ocone formula}}
\author{Deborah Schneider-Luftman}
\date{}
\maketitle

\begin{abstract}
In this paper we explore the fundamentals of the Martingale Representation Theorem (MRT) and a closely related result, the Clark-Ocone formula. We also investigate how far these theorems can be taken, notably beyond the regular Sobolev spaces, through changes of measures and enlargement of filtrations. We look at Brownian motion (B.M.) driven continuous martingales as well as Jump and L$\acute{e}$vy process-driven martingales \footnote{This article was written as a dissertation thesis, in partial fulfilment of the requirements for the MSc in Mathematics and Finance of Imperial College of Science, Technology and Medicine (Imperial College London, August 2012).}.
\end{abstract}


\section*{Introduction}
\addcontentsline{toc}{section}{Introduction}

The representation of martingales as an integral with respect to a fundamental process has been the subject of considerable research since the 1960's. Indeed, a key aspect of Financial Maths is to represent the evolution of  the prices of financial assets and of various portfolios. In particular considerable attention is given to designing hedges for investment portfolios made of options and other derivatives. To illustrate an application, we borrow the setting provided by Black and Scholes \cite{BlackScholes}, where the value $C_0$ of a contingent claim expiring at time T written on a specific payoff C is worth:
\[ C_0 = \mathbb{E}(C_T e^{-rT} \mid \mathcal{F}_0), \]
and,
\[C_t = \mathbb{E}(C_T e^{-r(T-t)} \mid \mathcal{F}_t). \]
We seek to create a hedge for $C_t$, therefore we build $V_t$ which is composed of a traded stock and bond $(S_t, B_t)$. In order to pick the right quantities $(\phi_t, \psi_t)$ for $(S_t, B_t)$ we need to know how $C_t$ is represented at t $\in [0,T]$. The aim is then to create a self-financing strategy where:
\[ C_t = V_t := \phi_t S_t + \psi_t B_t \mbox{ , } \forall t \leq T. \]
Both $S_t$ and $B_t$ have known representations, depending on the framework used. Typically,
\[ dS_t =  S_t \mu_t dt +  S_t \sigma_t dW_t \]
where $\mu_t$ and $\sigma_t$ are the parameters of the Geometric Brownian motion that $S_t$ follows, as assumed in the Black-Scholes framework. Additionally,
\[ dB_t = - rB_t dt. \]
So we have:
\begin{eqnarray*}
dV_t &=& \phi_t dS_t + \psi_t d B_t \\
&=& (\phi_tS_t\mu_t - \psi_trB_t)dt + \phi_tS_t\sigma_t dW_t.
\end{eqnarray*}
\\
But since the discounted value of $C_t$ is an $\mathcal{F}_t$-martingale:
\begin{equation} C_t e^{-rt} = \mathbb{E}(C_T e^{-rT} \mid \mathcal{F}_t), \label{introeq} \end{equation}
we can apply the Martingale Representation Theorem (see 1.1) to see that:
\[ C_t e^{-rt} = C_0 + \int^t_0 \pi_s dW_s \]
for some $\mathcal{F}$-predictable and square integrable $\pi_t$.
Then the following equality must hold:
\begin{eqnarray*}
C_t &=& V_t \mbox{  ,  } \forall t \leq T. \\
\mbox{So, } d C_t &=& dV_t \\
rC_tdt + e^{rt}\pi_t dW_t &=& (\phi_tS_t\mu_t - \psi_trB_t)dt + \phi_tS_t\sigma_t dW_t. 
\end{eqnarray*}
And thus, we now have a replicating strategy:
\begin{eqnarray*}
\phi_t &=& \frac{e^{rt}\pi_t }{S_t \sigma_t}, \\
\psi_t &=& \frac{\phi_tS_t\mu_t - rC_t}{rB_t}.
\end{eqnarray*}
While this result is a useful indication of how Hedging works, it raises a lot of questions and by itself cannot be applied directly. First of all, what is $\pi_t$? Without an explicit form for $\pi_t$, the above formula cannot be utilised in practical simulations. This is where the Clark-Ocone formula comes in. From the conventional form developed by Clark \cite{Clark}, the integrand $\pi_t$ has an explicit formula for continuous martingales contained within a Sobolev space $\mathbb{D}_{1,2}$ (see extensive definition in Chapter 2). The result can be extended to work beyond $\mathbb{D}_{1,2}$ and with martingales in various forms. Additionally the result \ref{introeq} states that $e^{-rt}C_t$ is a straightforward $\mathbb{P}$ /$\mathcal{F}$-martingale, however this is often not the case.\\ 
In order to achieve a condition like in \ref{introeq} we need to perform a change of measure, consequently the MRT and Clark-Ocone formula need to be adapted. In the same vein the MRT and Clark-Ocone formula can be modified in order to perform an enlargement of filtrations. For our uses this is required when we are looking at adding extra information to the base filtration $\mathcal{F}$ and hence creating a larger filtration $\mathcal{G}$. $\mathcal{F}$-martingales can be transformed into $\mathcal{G}$-martingales by adding a drift term $\mu$, which can be done in a similar way to changing measures with the Girsanov theorem. We will look at deriving a measure-valued adaptation of the MRT and Clark-Ocone formula to indentify the drift term $\mu$ and allow for the enlarged filtration $\mathcal{G}$ to exist.  \\
 \\ What justifies the use of the Martingale representation theorem within $L^2(\Omega, \mathcal{F}, \mathbb{P})$, for a given triple $(\Omega, \mathcal{F}, \mathbb{P})$? Why does it work and what does it indicate about the stochastic space on which traded assets evolve? One aim of this paper is to investigate the bottom-line facts that make integral representation possible for a given class of stochastic processes\\
\\ The example we have developed above is based on a continuous Brownian filtration. Despite their usefulness, works on Brownian filtrations are too restrictive as real financial assets are not just driven by continuous stochastic processes. In order to produce more realistic results, a large number of papers have focused their attention on understanding the representation of L$\acute{e}$vy and jump martingales. In this paper we will aslo explore the uses of the MRT and Clark-Ocone formula beyond continuous processes, notably through the work of Chou $\&$ Meyer \cite{choumeyer}, Boel, Wong $\&$ Varaiya \cite{Boel} and Davis \cite{Davis2}.\\ 
\\In the 1st section we review how to prove the MRT for continuous and general filtrations, in order to highlight the fundamental structure of the $L^2$ spaces within which financial assets evolve. In the 2nd section, we build on the proven results of section 1 to give an explicit representation to the integrand $\pi_t$ in the continuous and dis-continuous cases. In the 3rd section, we explore the generalization of the MRT and Clark-Ocone formula, notably through changes of measures and filtration \\
\\ Throughout this paper, we will assume knowledge of Mathematical finance explored in most introductionary courses. Primarily, we assume the readers have a background in stochastic calculus and familiarity with the definition of Stochastic bases, sigma-fields, filtrations, martingales and uniform integrability (U.I.). We also assume the readers have an understanding of the Girsanov theorem and Radon-Nikodym derivative, as well as of concepts of portfolio Hedging.

\section{The Martingale Representation theorem: Statement and Proof}

We begin this chapter by stating the theorem of interest in its most familiar form for Brownian martingales.

\subsection{Martingale representation theorem:  A regular Brownian motion approach }

Let $B=(B_t)$ be a Brownian motion on the stochastic base $(\Omega, \mathcal{F}, (\mathcal{F}_t), \mathcal{P})$ and $\mathcal{F}_t$ the augmented filtration generated by B. If X is square-integrable and $\mathcal{F}_{\infty}$-measurable, then \cite{wiki}: \\ \\ $\exists$  predictable process C adapted to $\mathcal{F}_t$ s. a.
\begin{eqnarray*}
X &=& \mathbb{E}(X) + \int^{\infty}_0 C_s dB_s , \\
\mbox{ and that: } \mathbb{E}(X|\mathcal{G}_t)& =& \mathbb{E}(X) + \int^{t}_0 C_s dB_s \mbox{, } \forall t \geq 0.
\end{eqnarray*}

\noindent We introduce the following notation which will be used throughout:
\[ \mathcal{M}^2 = \{ x_t \in L^2(\mathcal{F}_t,\mathbb{P}) \mid x_t \mbox{-u.i. martingale} \}. \]
For Martingales $(M_t) \in \mathcal{M}^2$ defined on $[0,T]$, the above result translates into the following with a 1-D Brownian motion:
\\ \\  $\exists$ predictable process C adapted to $F_t$ s. a.
 \begin{eqnarray*}
M_T &=& M_0 + \int^T_0 C_s dB_s , \\
\mbox{ and that: } M_t&=& \mathbb{E}(M_T|\mathcal{G}_t) = M_0 + \int^{t}_0 C_s dB_s \mbox{, } \forall t \geq 0.
\end{eqnarray*}

\noindent More generally, we use a multi-D Brownian motion as follow \cite{barnett}:

\begin{eqnarray*}
\exists n \mbox{ s.a. } B_t &=& \big\{ B^1_t,....,B^n_t\big\} \mbox{ and, } \forall X \in \mathcal{M}^2 \mbox{, we have:} \\
X_t &=& X_0 +  \int^t_0 \zeta_s \cdot dB_s \\
\mbox{with } \zeta_t &=& \big\{ \zeta^1_t,....,\zeta^n_t\big\} \mbox{ , } \zeta^i_t \mbox{-square integrable and predictable } \forall i.
\end{eqnarray*}


\subsection{Proofs: the fundamentals of integral representations in the $L^2$ spaces }

There are several approaches to proving the Martingale Representation Theorem in the continuous case as well as in the jump process sphere, but they tend to share common characteristics. Notably, it is generally agreed that these representation formulae originate from the fact that sets of integrals of specific processes are dense in $\mathcal{M}^2$. One of the most well-known works using this approach is from $\O$ksendal \cite{Oksendal1} and is as follows.

\subsubsection{$\O$ksendal's proof of the MRT in the continuous BM case}

Let $B := (B_t)$ be a n-dimensional Brownian Motion on the stochastic triple $(\Omega, \mathcal{F}, \mathbb{P})$, and $\mathcal{F}_t \subseteq \mathcal{F} $ be as follows:
\[ \mathcal{F}_t = \sigma \big\{ B_s(\omega) \mid s \leq t \big\} \]

We introduce a couple of definitions that will be used throughout:

\begin{defn}[{\bf Strong Orthogonality}]
 M, N $\in \mathcal{M}^2$ are strongly orthogonal if MN $\in \mathcal{M}^1$. \cite{Nualart}
\label{strongorthog}
\end{defn}

\begin{defn}[\bf{Weak Orthogonality}]
Two stochastic processes $X_t$ and $Y_t$, $t \in [0,T]$, are said to be weakly orthogonal when:
\[ \mathbb{E}(Y_T(\omega)X_T(\omega)) = 0 \mbox{ a.s.} \]
\end{defn}

\begin{thm} $\Psi$ = $\{ X_T \mid X_T = e^{\int^T_0 h(s) dW_s - \frac{1}{2} \int^T_0 h^2(s)ds}, h \in L^2([0,T]) \}$ is dense in $L^2(\mathcal{F}_T, \mathbb{P})$ \cite{Oksendal1}
\end{thm}

\begin{proof} This statement is true if $\forall g \in L^2(\mathcal{F}_T, \mathbb{P})$, g is not perpendicular to X for any X with $X_T \in \Psi$. We start by assuming that there is a g$\in L^2(\mathcal{F}_T, \mathbb{P})$ such as $g \perp X_T$. Strong orthogonality implies weak orthogonality so that means we have:
\begin{eqnarray*}
 \mathbb{E}(g(w)X_T(w)) &=& 0 \mbox{ a.s.} \\
 \int_{\Omega} e^{\int^T_0 h(s) dW_s(w) - \frac{1}{2} \int^T_0 h^2(s)ds} g(w) d\mathbb{P} &=& 0 \\
e^{- \frac{1}{2} \int^T_0 h^2(s)ds}  \int_{\Omega} e^{\int^T_0 h(s) dW_s(w) } g(w) d\mathbb{P} &=& 0.
\end{eqnarray*}
In particular, $\forall \lambda \in \mathbb{R}^n$ we have the following:
\[ G(\lambda):=  \int_{\Omega} e^{\sum^n_i \lambda_i W_{t_i} } g(w) d\mathbb{P} = 0. \]
$G(\lambda)$ is real analytic in $\mathbb{R}^n$, and thus has an analytic extension to the complex space $\mathbb{C}^n$:
\[ G(z):=  \int_{\Omega} e^{\sum^n_i z_i W_{t_i} } g(w) d\mathbb{P} = 0 \mbox{ , } \forall z \in \mathbb{C}^n \]

We know from \cite{Oksendal1} that $C^{\infty}_0(\mathbb{R}^n)$ is dense in $L^2(\mathcal{F}_T, \mathbb{P})$. However, $\forall \phi \in C^{\infty}_0(\mathbb{R}^n)$, we have:

\begin{eqnarray*}
\mathbb{E}(\phi(W)g(w)) &=&  \int_{\Omega} \phi(W_{t_1},..., W_{t_n}) g(w) d\mathbb{P} \\
&=& \int_{\Omega} (2\pi)^{-n/2} \Big( \int_{\mathbb{R}^n} \hat\phi(y)e^{iW\cdot y}dy \Big) g(w) d\mathbb{P} \mbox{ (Inverse Fourier transform)} \\
&=&  (2\pi)^{-n/2} \int_{\mathbb{R}^n} \hat\phi(y) \Big( \int_{\Omega} e^{iW \cdot y} g(w) d\mathbb{P} \Big) dy \mbox{ (Fubini's Theorem)} \\
&=&  (2\pi)^{-n/2} \int_{\mathbb{R}^n} \hat\phi(y) \Big( \int_{\Omega} e^{ i(\sum_i W_iy_i)} g(w) d\mathbb{P} \Big) dy \\
&=&  (2\pi)^{-n/2} \int_{\mathbb{R}^n} \hat\phi(y)  G(iy) dy = 0.
\end{eqnarray*}
This can only be true iff g$\equiv$0. 
\end{proof}

Define analogously the following spaces:
\[ \Psi_t = \{ X_t \mid X_t = e^{\int^t_0 h(s) dW_s - \frac{1}{2} \int^t_0 h^2(s)ds}, h \in L^2([0,T]) \} \mbox{ dense in } L^2(\mathcal{F}_t,\mathbb{P}) \]

$\forall Y_t(w) \in \Psi_t$, it is easy to see that:
\[ dY_t = Y_t h(t) dW_t  \mbox{ (Ito's lemma), }\]
hence $\exists h(t) \in L^2([0,T])$

\begin{eqnarray*}
Y_t &=& Y_0 + \int^t_0 Y_s h(s) dW_s \\
&=& 1 + \int^t_0 Y_s h(s) dW_s \\
&\mbox{ or }& \mathbb{E}(Y_t) + \int^t_0 Y_s h(s) dW_s.
\end{eqnarray*}

As $\Psi$ is dense in $L^2(\mathcal{F}_T, \mathbb{P})$, we can approximate any element of $L^2(\mathcal{F}_T, \mathbb{P})$ with suitable convergent elements of $\Psi$. So $\forall X \in L^2(\mathcal{F}_T), \exists (\psi_n) \in \Psi$ such as $\psi_n \rightarrow X$ in $L^2(\mathcal{F}_T)$. Since any convergent sequence  $\psi_n \in \Psi$ is also a cauchy sequence, $(\psi_n)$ is cauchy in $L^2(\mathcal{F}_T)$:
\begin{equation} \mathbb{E}( (\psi_n - \psi_m)^2) \rightarrow 0 \mbox{ as }m,n \rightarrow +\infty,
\label{cauchy} \end{equation}

where we have:
\[ \psi_n = \mathbb{E}(\psi_n) + \int^T_0 f_n(s,\omega)dW_s.\]
The function $f_n$ is defined as
\[ f_n(t,\omega) := \psi_n(\omega)h_n(t), \] 
and $f_n(T,\omega) \in L^2([0,T] \times \Omega)$.
So 
\begin{eqnarray*}
\mathbb{E}( (\psi_n - \psi_m)^2) &=& \mathbb{E}\big( (\mathbb{E}( \psi_n - \psi_m) + \int^T_0 f_n-f_m dW_s)^2 \big) \\
&=& \mathbb{E}( \psi_n - \psi_m)^2 + \int^T_0 \mathbb{E}((f_n-f_m)^2) ds \mbox{ (Ito's isometry)}\\
& \geq & \int^T_0 \mathbb{E}((f_n-f_m)^2) ds.
\end{eqnarray*} 

Result \ref{cauchy} implies that $\forall A_T \in L^2(\mathcal{F}_T, \mathbb{P})$, $\exists f(t) \in L^2(\mathbb{R} \times \Omega), \psi_n \in \Psi$ such as $ f_n \rightarrow f$ as $n \rightarrow \infty$ where:
\begin{equation} A_t = \lim_{n \rightarrow \infty} \psi_n = \mathbb{E}(A) + \int^t_0 f(s) dW_s \label{itorepth}. \end{equation}

Result \ref{itorepth} is knowns as the {\bf Ito Representation Theorem}. \cite{Oksendal1}
Indeed, in this framework the function f(t) is unique: Assume there are 2 functions $f_1(t)$ and $f_2(t)$ satisfying \ref{itorepth}, then:
\begin{eqnarray*}
A_t(w) = \mathbb{E}(A) + \int^t_0 f_1(s) dW_s = \mathbb{E}(A) + \int^t_0 f_2(s) dW_s \\
\Longrightarrow 0=(A_t-A_t)^2= \int^t_0 \mathbb{E}((f_1(s)-f_2(s))^2) ds \\
\Longrightarrow f_1(t) = f_2(t) \mbox{ a.s. for all t} \in [0,T].
\end{eqnarray*}

At this stage, the martingale representation theorem is just an application of the Ito representation theorem (equation \ref{itorepth}) to $\mathcal{M}^2([0,T]) \subset L^2(\mathcal{F}_T, \mathbb{P})$:
\[ \forall A \in \mathcal{M}^2([0,T]) \mbox{ ,} \exists ! f(t)  \in L^2((\mathcal{F}_T, \mathbb{P}) \times [0,T]) \mbox{ s.a. } A_t = \mathbb{E}(A_T) + \int^t_0 f(s) dW_s. \]

The dense set approach combined with the orthogonality argument, beyond giving a representation formula for martingales, makes a general statement about the $L^2(\mathcal{F}_T, \mathbb{P})$ space. In this setting, $\mathcal{M}^2$ and $L^2(\mathcal{F}_T, \mathbb{P})$ are Hilbert spaces under a suitable inner product:

\[  L^2(\mathcal{F}_T, \mathbb{P}) = \bigoplus^{\infty}_{i=1} \mathcal{S}(M_i), \]
where
\[ \mathcal{S}(M_i) = \big\{ \int_{[0,T]} \phi dM_i \mid \phi \in L^2( < M_i >) \big\} \mbox{ , } \exists M_i \in L^2(\mathcal{F}_T, \mathbb{P}),  i  \in \mathbb{N}.\]
$\mathcal{S}(M_i) $ is also known as the Stable Subspace generated by $M_i$. (\cite{Davis1}) \\

This approach is widely explored  elsewhere in the literature, notably in Dellacherie \cite{Dellacherie}, Yor \cite{Yor} and  Kunita $\&$  Watanabe \cite{KunitaWatanabe}.\\
While this idea and framework is very dominant in works on continuous martingales, it does not limit itself to the extends of the Brownian Filtration. This is indeed the case as shown in the work of L$\o$kka \cite{Lokka2} and Davis \cite{Davis2}.

\subsubsection{MRT for martingales driven by Poisson and pure-jump processes }
\paragraph{setting}

In this section, we bring our attention to a specific stochastic process, the Compensated Poisson process:
Let $N_t$ be a poisson process of intensity $\lambda$ and define:

\begin{eqnarray*}
\bar{N}_t &=& N_t - \lambda t.\\
\end{eqnarray*}

$\bar{N}_t $ has the following characteristic function:

\begin{eqnarray*}
\phi_{\bar{N}_t}(z) &=& e^{\lambda t (e^{iz}-1-iz)}.\\
\end{eqnarray*}
Additionally:

\begin{eqnarray*}
\mathbb{E}[\bar{N}_t | \bar{N}_s] &=& \mathbb{E}[ N_t - \lambda t | \bar{N}_s] \\
&=& \mathbb{E}[ N_t - N_s  | \bar{N}_s] - \lambda t + N_s \\
&=& N_s - \lambda t + \lambda (t-s) \\
&=& \bar{N}_s.
\end{eqnarray*}

So $\bar{N}_t$ is a martingale: $ \forall t > s$ , $  \mathbb{E}[\bar{N}_t | \bar{N}_s] =  \bar{N}_s$. \\
When rescaled, it has similar properties to a regular Brownian motion:

\begin{eqnarray*}
\mathbb{E}\Big[ \frac{\bar{N}_t }{\lambda} \Big] &=& 0\\
\mathbb{V}\Big[ \frac{\bar{N}_t }{\lambda} \Big] &=& t,
\end{eqnarray*}

hence $\bar{N}_t $ is a likely candidate for the integrator in the MRT, especially since it has the following curious property \cite{poisson}:

\begin{eqnarray*}
\left( \begin{array}{c} 
\bar{N}_t  \\
\lambda \end{array} \right) \rightarrow (W_t) \mbox{ as } \lambda \rightarrow \infty.
\end{eqnarray*}

As we will see, this type of jump processes is very important to the representation of pure-jump martingales, as the use of Poisson-generated filtrations is very common in the literature on applications to hedging strategies. Indeed, as we will show below, martingales evolving on a stochastic base can be conveniently represented in an integral form using compensated poisson processes.\\

We base ourselves on the work developed in L$\o$kka (2005) \cite{Lokka2} and start with $([0,T],\mathcal{B},\lambda)$ where $\mathcal{B}$ is the Borel $\sigma$-algebra and $\lambda$ is a Radon-measure that charges all open sets and that is diffuse over $\mathcal{B}$. Then we set:
\begin{eqnarray*}
\Omega &=& \{ \omega = \sum^n_j \delta_{t_j} \mid n \in \mathbb{N} \cup \{ \infty\}, t_j \in [0,T] \}, \\
\mathcal{F}_0 &=& \sigma\{ \omega(A) \mid A \subseteq \mathcal{B} \}, \\
\mathbb{P} &:& \mbox{ Probability measure s.a. } t \mapsto \omega([0,t]) \mbox{ is a Poisson process,} \\
\mathcal{F} &=& \mathbb{P}\mbox{-completion of } \mathcal{F}_0.\\
\end{eqnarray*}

On this new stochastic base we introduce L, a square-integrable Poisson jump process defined as:
\[ L_t = \int^t_0 \int_{\mathbb{R}_0} z \tilde{N}(dz,dt). \]
$\tilde{N}(z,t)$ is a compensated poisson process. L can be otherwise noted as follow:
\[ L_t = \int^t_0 \int_{\mathbb{R}_0} z (\mu-\pi)(dz,dt),\]
where $\mu-\pi$ is the measure of the compensated poisson process $\tilde{N}(z,t)$. $\mu$ is the Poisson measure of the process $(L_t)$:
\[ \mu(\Lambda, \Delta t) = \sum_{s \in \Delta t} 1_{\Lambda}(\Delta L_s) \]
and $\pi$ is the compensator of $\mu$ and can be defined by the following relation:
\[ \mathbb{E}(L^2_t) = \int_{\mathbb{R}_0} z^2 \pi(dz,dt). \]
$\pi$ is known to have the form $\pi(dz,dt) = \upsilon(dz)dt$, where $\upsilon$ is the L$\acute{e}$vy measure of L. \cite{Lokka2} \\
At this stage, we can consider the complete filtered probability space $(\Omega, \mathcal{F}, (\mathcal{F}_t), \mathbb{P})$ by setting the following:
\[ \mathcal{F}_t = \sigma\{ L_s \mid s \leq t \}. \] 

\paragraph{Representation results}

Now that we have discussed the setting of the underlying stochastic base $(\Omega, \mathcal{F}, (\mathcal{F}_t), \mathbb{P})$ , we can state and prove several results that will then lead to an integral representation formula for martingales in $L^2(\mathcal{F}_T,\mathbb{P})$.

\begin{thm} 
$\mathcal{A} = \{ \phi(L_{t_1},...,L_{t_n}) \mid t_i \in [0,T], t_0=0, t_n=T, \phi \in C^{\infty}_0(\mathbb{R}^n), n \in \mathbb{N} \}$  \\ is dense in $L^2(\mathcal{F}_T,\mathbb{P})$.
\label{Lokka2thm1}
\end{thm}

\begin{proof} Consider a set $\{t_i\}_{\mathbb{N}}$ dense in [0,T], and a nested set of $\sigma$-fields $\mathcal{G}_n = \sigma(L_{t_1},...,L_{t_n})$. Clearly, $\mathcal{G}_n \subset \mathcal{G}_{n+1} \forall n \in \mathbb{N}$. Also, $\mathcal{F}_T$ is the smallest $\sigma$-field containing all $\mathcal{G}_n$ for any n in $\mathbb{N}$ \cite{Lokka2}. Then, $\forall g \in L^2(\mathcal{F}_T,\mathbb{P})$,

\[ g = \mathbb{E}(g \mid \mathcal{F}_T) = \lim_{n \rightarrow \infty} \mathbb{E}( g \mid \mathcal{G}_n). \]
So by the Doob-Dynkin Lemma  we have :

\[ \exists g_n : \mathbb{R}^n \mapsto \mathbb{R} \mbox{ borel-measurable s.a. } \mathbb{E}(g \mid \mathcal{G}_n) = g_n(L_{t_1},...,L_{t_n}). \]
$g_n$ can be approximated in $L^2(\mathcal{F}_T, \mathbb{P})$ with $\phi_{n_m} \in C^{\infty}_0(\mathbb{R}^n)$ \cite{Folland}  where 

\[ \| g_n(L_t) - \phi_{n_m}(L_t) \|_{L^2(\mathbb{P})} \rightarrow 0 \mbox{ as } m \rightarrow \infty. \]
Hence, every g in $L^2(\mathcal{F}_T,\mathbb{P})$ can be approximated with a function $\phi \in C^{\infty}_0(\mathbb{R}^n)$.
\end{proof}

While this result is useful, we are looking to make a more precise statement about the structure of $L^2(\mathcal{F}_T, \mathbb{P})$. To do so, we start with the result below which is an adaptation of the regular It$\hat{o}$ Isometry;

\begin{thm}
 Whenever $\phi: \phi(t,z,\omega)$ is a predictable process, the below is true:
\[ \mathbb{E}\big( \big( \int^T_0 \int_{\mathbb{R}_0} \phi(t,z) (\mu-\pi)(dz,dt) \big)^2 \big) = \mathbb{E}\big( \int^T_0 \int_{\mathbb{R}_0} \phi^2(t,z) \pi(dz,dt) \big). \]
\label{Poissonitoisometry}
\end{thm}
One implication of this result is that $\int^T_0 \int_{\mathbb{R}_0} \phi(t,z) \tilde{N}(dz,dt) \in L^2(\mathbb{P})$ if $\phi \in L^2(\pi \times \mathbb{P})$. \\

Now set a new continuous function $\gamma:\mathbb{R}_0 \mapsto (-1,0) \cup (0,1)$ :

\begin{equation}
 \gamma(z) = \left\{
      \begin{array}{cc}
        e^z - 1& z < 0 \\
        1 - e^{-z} & z > 0
      \end{array}
      \right. \nonumber
\end{equation}

One nice property of $\gamma$ is that it approaches zero as fast as z does. Indeed:

\begin{eqnarray*}
\lim_{z \rightarrow 0} \frac{\gamma(z)}{z} &=& \left| \begin{array}{c} \lim_{z \rightarrow 0^+} \frac{1-e^{-z}}{z}  \\ \lim_{z \rightarrow 0^-} \frac{e^z - 1}{z}  \end{array} \right. \\
&=&  \left| \begin{array}{c} \lim_{z \rightarrow 0^+} e^{-z}  \\ \lim_{z \rightarrow 0^-} e^z  \end{array} \right. \mbox{ (Hospital rule)} \\
&=& 1
\end{eqnarray*}

This combined with the facts that $\upsilon\big( (-\infty, -1] \cup [1, \infty) \big) < \infty$ and that $\gamma$ is bounded implies that :
\[ e^{\gamma \lambda} - 1 \in L^2(\upsilon) \mbox{ , } \forall \lambda \in \mathbb{R} \]

and that $\forall h \in C([0,T])$,
\begin{eqnarray}
 e^{\gamma \lambda} -1 \in L^2(\pi) , \nonumber \\
 h \gamma \in L^2(\pi) \nonumber \\
e^{h\gamma} - 1 -h\gamma \in L^1(\pi) \label{conditionsLokka2}
\end{eqnarray}

The use of the function $\gamma$ and of the results \ref{conditionsLokka2} are necessary for us to proceed as $\exists \lambda \in \mathbb{R}$ such as $ e^{L_t \lambda} \notin L^2(\mathbb{P})$. Without these conditions, it is not possible to find an integral representation for elements of $L^2(\mathcal{F}_T,\mathbb{P})$. However, modifying the jumps of L with $\gamma$ ensures that the exponential of the modification lies in $L^2(\mathcal{F}_T,\mathbb{P})$.

\begin{thm}
The linear span of random variables of the form:
\begin{equation} \exp{ \Big\{ \int^T_0\int_{\mathbb{R}_0} h(t)\gamma(z)\tilde{N}(dz,dt) -  \int^T_0\int_{\mathbb{R}_0}(e^{h(t)\gamma(z)}-1-h(t)\gamma(z))\pi(dz,dt) \Big\} } \nonumber
\end{equation}
where h $\in C([0,T])$ is dense in $L^2(\mathcal{F}_T,\mathbb{P})$.
\label{lemma2lokka2}
\end{thm}

\begin{proof} Set 
\[ \tilde{L}_t = \int^t_0 \int_{\gamma(\mathbb{R}_0)} \gamma(z)(\mu-\pi)(dz,dt). \]

$\tilde{L}_t$ is a Poisson process with a poisson random measure $\tilde{\mu}(\gamma^{-1}(\Lambda),t)$. Clearly, we see that:
\[ \sigma(\tilde{L}) \subseteq \sigma(L). \]

Now introduce:
\[ \hat{L}_t = \int^t_0 \int_{\gamma(\mathbb{R}_0)} \gamma^{-1}(z)(\tilde{\mu}-\tilde{\pi})(dz,dt). \]

Again, $\sigma(\hat{L}) \subseteq \sigma(\tilde{L})$. $\hat{L}$ is also a Poisson process with poisson random measure $\hat{\mu}$:
\[ \hat{\mu}(\Lambda,t) = \tilde{\mu}(\gamma(\Lambda),t) = \mu(\gamma(\gamma^{-1}(\Lambda)),t) = \mu(\Lambda,t). \]

That implies that $\hat{L}=L$ and that $\sigma(\hat{L}) = \sigma(L)$. Hence we have $\sigma(\tilde{L}) = \sigma(L)$. \\
Take $Y_T$ be of the form stated in Theorem \ref{lemma2lokka2}. Then:

\[ dY_t = \int_{\mathbb{R}_0} Y_{t-} (e^{h(t)\gamma(z)}-1)(\mu-\pi)(dt,dz) \mbox{ (It$\hat{o}$'s formula).} \]
By the It$\hat{o}$ isometry (Theorem \ref{Poissonitoisometry}) and the continuity of the law of L, we have that \cite{Lokka2}
\[  \mathbb{E}(Y_t^2) = \exp{ \big( \int^t_0 \int_{\mathbb{R}_0} (e^{h(s)\gamma(z)}-1)^2\pi(dz,dt) \big)} \]
and
\[ \| Y_t\|^2_{L^2(\mathbb{P})} = \exp{ ( \| e^{h\gamma}-1 \|^2_{L^2(\pi)} ) }.\]

Since we know that $e^{h\gamma}-1  \in L^2(\pi)$, we deduce that $Y_T \in L^2(\mathbb{P})$. Additionally, since $e^{h\gamma}-1-h\gamma  \in L^1(\pi)$, we have that:

\[ e^{\int^T_0 h(t)d\tilde{L}_t} \in L^2(\mathbb{P}) \mbox{ , } \forall h \in C([0,T]), \]
and more specifically, for any $\lambda \in \mathbb{R}^n$ and $\{t_i\} \in [0,T]$,

\[ e^{\lambda \cdot \tilde{L}_t} \in L^2(\mathbb{P}). \]

The above then allows us proceed with the statement of theorem \ref{lemma2lokka2}. From here onwards, we follow the same argument explored with the B.M. in $\O$ksendal \cite{Oksendal1}: we assume that there is a g $\in L^2(\mathcal{F}_T,\mathbb{P})$ that is orthogonal to all variables of the form stated in theorem \ref{lemma2lokka2}.  Through a very similar argument, we find that for any $\phi \in C^{\infty}_0(\mathbb{R}^n)$,
\[ \int_{\Omega} \phi(\tilde{L}_{t_1},...,\tilde{L}_{t_n})gd\mathbb{P} = 0.\]

However, we also know from theorem \ref{Lokka2thm1} that the variables of type $\phi(\tilde{L}_{t_1},...,\tilde{L}_{t_n})$ are dense in $L^2(\mathcal{F}_T,\mathbb{P})$. Hence we must deduce that g$\equiv$0.
\end{proof} 

Therefore, based on Theorem \ref{lemma2lokka2}, we can prove the following representation theorem just as done in section 1.2.1.

\begin{thm}
$\forall F \in L^2(\mathcal{F}_T,\mathbb{P})$,  there is a unique process $\psi \in L^2([0,T] \times \mathbb{R}_0 \times \Omega, \pi \times \mathbb{P})$ such as:
\[F = \mathbb{E}(F) + \int^T_0 \int_{\mathbb{R}_0} \psi(t,z)(\mu-\pi)(dz,dt), \]
or
\[ F = \mathbb{E}(F) + \int^T_0 \int_{\mathbb{R}_0} \psi(t,z) \tilde{N}(dz,dt). \]
\label{lokkaMRTpoisson}
\end{thm}

Therefore for $\forall (m_t) \in \mathcal{M}^2$ where $m_{\infty} < \infty$  , $\exists ! \psi \in L^2([0,T] \times \mathbb{R}_0 \times \Omega, \pi \times \mathbb{P})$ such that we have:
\[ m_t = \mathbb{E}(m_{\infty} \mid \mathcal{F}_t) = \int^t_0 \int_{\mathbb{R}_0} \psi(t,z) \tilde{N}(dz,dt). \]

\paragraph{Extension to other jump processes}

The above representation result involves Compensated Poisson jump processes, but it is extendable to other type of jump processes. This is indeed what the work of Davis\cite{Davis2} shows.\\
 We start with a basic jump process ($x_t$) defined on a Blackwell space $(X, \mathcal{L})$. 
We remind the reader of the following definition: \\

\begin{defn}[{\bf Blackwell Spaces}]
 A measurable space (X,$\mathcal{A}$) is a Blackwell space if the $\sigma$-algebra $\mathcal{A}$ it is the smallest countably-generated $\sigma$-field and contains all singletons. In other words, $\exists$ countable sets $\{A_i \mid i \in \mathbb{N} \}$ that generate $\mathcal{A}$. This type of set is otherwise known as Lusin spaces, as introduced in Blackwell \cite{Blackwell}. \\
One important property is that any one-to-one $\mathcal{A}$-measurable mapping f: X $ \rightarrow (Y,\mathcal{B})$, where $\mathcal{B}$ is also countably generated, is an isomorphism. \cite{Bog}
\end{defn}
Now define $ (Y, \mathcal{Y}) = \Big( (\mathbb{R}^+\times X)\bigcup \{(\infty,z_{\infty})\}, \sigma\{\mathcal{B}(\mathbb{R}^+) \ast \mathcal{L},  \{(\infty,z_{\infty})\} \} \Big)$ and its copies $(Y^i, \mathcal{Y}^i)$, $i \in \mathbb{N}$, $z_0$, $z_{\infty}$ are fixed elements of X. Now set:

\begin{eqnarray*}
\Omega &=& \prod^{\infty}_{i=1} Y^i, \\
\mathcal{F}^0 &=& \sigma \Big\{ \prod^{\infty}_{i=1} \mathcal{Y}^i \Big\}, \\
\mbox{Let } (S_i,Z_i): \Omega \rightarrow Y^i &&\mbox{ be the coordinate mapping, and } \\
\omega_k : \Omega \rightarrow \Omega_k = \prod^k_{i=1} Y^i &\mbox{ s.a. }& \omega_k(w) = (S_1,Z_1,...,S_k,Z_k).
\end{eqnarray*}

We can now define:
\begin{eqnarray*}
T_k(w) = \sum^k_{i=1} S_i(w), \\
T_{\infty}(w) = \lim_{k \rightarrow \infty} T_k(w),
\end{eqnarray*}

and the jump process of interest:\\

\[x_t(\omega)= \left\{ \begin{array}{cc} z_0, & \mbox{if } t<T_1(\omega) \\ Z_i(\omega), & \mbox{if } t \in [T_i,T_{i+1}) \\ z_{\infty}, & \mbox{if } t \geq T_{\infty} \end{array} \right. .\]

This random variable generates an increasing family of $\sigma$-fields $(\mathcal{F}^0_t)$ where $\mathcal{F}^0_t = \sigma\{ x_s \mid s \leq t\}$. Then $\mathcal{F}_t(\mathcal{F})$ is the $\mathbb{P}$-augmented $\sigma$-field of $\mathcal{F}^0_t(\mathcal{F}^0)$ .\\
To go with $(\Omega, \mathcal{F}^0)$, we define the following family of probability measures $(\mu_i)$: for $\Gamma \in \mathcal{Y}$ and $\eta \in \Omega_{i-1}$
\begin{eqnarray*}
P[ (T_1,Z_1) \in \Gamma] = \mu^1(\Gamma),\\
P[ (T_i,Z_i) \in \Gamma |  \omega_{i-1} = \eta ] = \mu^i(\eta ; \Gamma).
\end{eqnarray*}
To ensure that the jump times of $x_t$ do not occur at once, we impose: 
\[ \mu^i(\omega_{i-1} ; (\{0\}\times X) \cup (\mathbb{R}^+ \times \{Z_{i-1}\})) = 0\mbox{ } \forall i \in \mathbb{N}. \]

There is a fundamental family of processes associated with the filtration generated by $(x_t)$. For A $\in \mathcal{L}$ and t $\in \mathbb{R}^+$ set:

\begin{eqnarray*}
p(t,A)(w) &=& \sum_i 1_{(t \geq T_i)}1_{(Z_i \in A)}, \\
\tilde{p}(t,A) &=& - \sum^{j-1} \int^{S_i}_0 \frac{1}{F^i_{s-}}dF^{iA}_s - \int^{t-T_{j-1}}_0 \frac{1}{F^i_{s-}}dF^{iA}_s , \mbox{ for } t \in (T_{j-1},T_j] \\
\mbox{where } && F^{iA}_t = \mu^i(\omega_{i-1};(t, \infty] \times A) \mbox{ and } F^i_t = P(T_i > t), \\
q(t,A) &=& p(t,A) - \tilde{p}(t,A).
\end{eqnarray*}

\begin{thm} 
 for fixed k and A $\in \mathcal {L}$, $q(t \wedge T_k, A)$ is an $\mathcal{F}_t$-martingale \cite{Davis2}
\end{thm}

\begin{proof}
To address the above, we use the proof to an analogous proposition in Chou $\&$ Meyer \cite{choumeyer}:

\begin{eqnarray*}
q( t \wedge T_k, A) &=& p(t \wedge T_k,A) - \tilde{p}(t \wedge T_k,A) \\
&=& \sum_i 1_{(t \wedge T_k \geq T_i)}1_{(Z_i \in A)} + \sum^{j-1} \int^{S_i}_0 \frac{1}{F^i_{s-}}dF^{iA}_s + \int^{t-T_{j-1}}_0 \frac{1}{F^j_{s-}}dF^{jA}_s \\
&=& \sum^{j-1}_i \Big[ 1_{(t  \geq T_i)}1_{(Z_i \in A)} +  \int^{S_i}_0 \frac{1}{F^i_{s-}}dF^{iA}_s \Big] + 1_{(t  \geq T_j)}1_{(Z_j \in A)} +  \int^{t-T_{j-1}}_0 \frac{1}{F^j_{s-}}dF^{jA}_s \mbox{, where j $\leq$ k}
\end{eqnarray*}

Set $q_i(t,A) =  1_{(t  \geq T_i)}1_{(Z_i \in A)} +  \int^{S_i}_0 \frac{1}{F^i_{s-}}dF^{iA}_s $ $\forall i <j$, and $q_j(t,A) =  1_{(t  \geq T_j)}1_{(Z_j \in A)} +  \int^{t-T_{j-1}}_0 \frac{1}{F^j_{s-}}dF^{jA}_s$. Then
\[ q( t \wedge T_k, A) = \sum^j_i q_i(t,A) \mbox{, where $t \in [T_{j-1};T_j)$ whenever j $\leq$ k} \]
It can then be shown that each element $q_i(t,A)$ is a martingale.

\end{proof}

We define the following class of integrands:
\[ \mathcal{P} = \Big\{ g: \Omega \times Y \rightarrow \mathbb{R} \mid g(t,z,\omega) = g^k(\omega_{k-1}; t,z)1_{(T_{k-1},T_k]} \Big\}, \]
and:
\begin{eqnarray*}
L^1(p) = \Big\{ g \in \mathcal{P} \mid \mathbb{E}\big( \int_{\mathbb{R}^+ \times X} |g(t,z)|p(dt,dz)\big) < \infty \Big\}, \\
L^1_{loc}(p) = \Big\{ g \in \mathcal{P} \mid g1_{(t<\sigma_k)} \in L^1(p), k \in \mathbb{N}, \exists \mbox{ stopping times } \sigma_k \uparrow T_{\infty} \Big\}.
\end{eqnarray*}
$L^1(\tilde{p})$ and $L^1_{loc}(\tilde{p})$ are defined analogously.\\
Then for $g \in L^1_{loc}(p) $ and $ M^g_t  = \int_{(0,t] \times X} g(s,z) q(ds,dz)$, there is a sequence of stopping times $\tau_n \uparrow T_{\infty}$ where $M^g_{t \wedge \tau_n}$ is a u.i. martingale for all n $\in \mathbb{N}$.(\cite{Davis2}). This induces the following to be true:
\begin{equation}
M_t = M_{t \wedge T_1} + \sum^{\infty}_{k=2} \big( M_{t \wedge T_k}-M_{T_{k-1}}\big)1_{(t \geq T_{k-1})}
\label{uimgs}
\end{equation}
whenever $M_t$ is a u.i. martingale. To simplify \ref{uimgs} further, we set:
\begin{eqnarray*}
X^1_t &=& M_{t \wedge T_1}, \\
X^k_t &=&  M_{(t+T_{k-1}) \wedge T_k}-M_{T_{k-1}}.
\end{eqnarray*}
Then $M_t$ in \ref{uimgs} is transformed into:

\begin{equation}
M_t = \sum^{\infty}_{k=1} X^k_{(t-T_{k-1})\vee 0}.
\label{newmt}
\end{equation}

For fixed k and t$\geq$0, set $ \mathcal{H}_t = \mathcal{F}_{(t+T_{k-1}) \wedge T_k}$. According the the optimal sampling theorem, $(X^k_t)$ is a $\mathcal{H}_t $-martingale. In other words, there is a function $h^k$ such as:

\[ X^k_t = \mathbb{E}(h^k(\omega_{k-1};S_k,Z_k) \mid \mathcal{H}_t ).\mbox{ \cite{Davis2} } \]

Using proposition 5 of \cite{Davis2} and proposition 2 of \cite{choumeyer} we see that each u.i. martingale $X^k_t$ can be expressed as follow:
\begin{eqnarray*}
X^k_t = 1_{(t \geq T_k)}h^k(\omega_{k-1};S_k,Z_k) - 1_{(t < T_k)}\frac{1}{F^{k}_t} \int_{(0,t] \times X} h^k(\omega_{k-1};S_k,Z_k) dF^{k}(s,z).
\end{eqnarray*}
Working on the interval $[0,T_k)$, that gives:

\begin{eqnarray*}
X^k_t =  \frac{-1}{F^k_t}  \int_{(0,t] \times X} h^k(\omega_{k-1};S_k,Z_k) dF^{k}(s,z).
\end{eqnarray*}

Further calculations (\cite{Davis2}) then lead to the following end result:
\begin{equation}
X^k_t = \int_{(0,t]\times X} g^k(\omega_{k-1};s,z)q^k(ds,dz),
\label{xt}
\end{equation}

where $q^k(t,A)= q((t+T_{k-1}) \wedge T_k,A)$. Then \ref{xt} turns \ref{newmt}  into:

\begin{eqnarray}
M_t &=& \sum^{\infty}_{k=1} X^k_{(t-T_{k-1})\vee 0} \nonumber \\
 &=& \sum^{\infty}_{k=1} \int_{(0,(t-T_{k-1}) \vee 0]\times X} g^k(\omega_{k-1};s,z)q^k(ds,dz) \nonumber\\
&=&  \int_{(0,t] \times X} g(s,z)q(ds,dz).
\label{jumpmrt}
\end{eqnarray}

The function g $ \in \mathcal{P}$ is defined by the collection $\{ g^k \mid k \in \mathbb{N} \}$ such that \ref{jumpmrt} holds. This resulting function g $\in L^1_{loc}(p) $, as explained in \cite{Davis2}.\\
\\
We have seen so far that the Martingale representation theorem is applicable to a diverse range of probability spaces and martingales, but as the work of Nualart and Schoutens \cite{Nualart} in has shown, it can reach to much more general martingales.

\subsubsection{Nualart and Schoutens: MRT using general L$\acute{e}$vy processes}

In this setting we start with $X = \{ X_t \mid t \geq 0 \}$ being a L$\acute{e}$vy process defined on a complete stochastic triple $(\Omega, \mathcal{F}, \mathbb{P})$. $\mathcal{F}$ is the $\sigma$-field generated by X. We let:
\[ \mathcal{F}_t = \mathcal{G}_t \vee \mathcal{N}, \]
where
\[ \mathcal{G}_t = \sigma\{ X_s \mid s \leq t \} \]
and 
\[ \mathcal{N} = \{ A \mid A-\mathbb{P} \mbox{-null set of } \mathcal{F} \}. \]
Then we can introduce compensated power jump processes as follow:
\begin{itemize}
\item $X = (X_t)$ is a L$\acute{e}$vy process, i.e. X has stationary and independent increments, $X_0=0$ and is c$\grave{a}$dl$\grave{a}$g,
\item $X_{t-} = \lim_{s \rightarrow t, s<t} X_s$,
\item $\Delta X_t = X_t - X_{t-}$, 
\item $ X^i_t = \sum_{s \leq t} ( \Delta X_s)^i $ ,  and $X^1_t = X_t$,
\item $ Y^i_t = X^i_t - \mathbb{E}(X^i_t) := X^i_t - m_it$ are then a martingale called Teugels martingales \cite{Nualart} 
\end{itemize}

In here, $ \forall (m_t) \in \mathcal{M}^2, m_t = \mathbb{E}(m_{\infty} \mid \mathcal{F}_t) $, and $m_i$ are the ith central moments of the process X. \\
We can start with the case of i =1. There we have:
\[ Y^1_t = X_t - m_1t. \]
We can see that for any $s < t \leq T$ :

\begin{eqnarray*}
\mathbb{E}(Y^1_t  \mid \mathcal{F}_s) &=& \mathbb{E}( X_t - m_1t \mid \mathcal{F}_s)\\
&=&  \mathbb{E}( X_t - X_s\mid \mathcal{F}_s) -m_1t+X_s \\
&=& X_s-m_1t + m_1(t-s) \mbox{ (independent increments of L$\acute{e}$vy processes)}\\
&=& Y^1_s.
\end{eqnarray*}

Hence $Y^1$ is indeed a martingale. $Y^i$ is also a martingale for $i>1, i \in \mathbb{N}$, as shown in \cite{Schoutens}.\\
We introduce the following space:
\begin{equation}
\mathcal{H}= \Big\{ \int^{\infty}_0 f(t)dY^1_t \mid  f \in L^2(\mathbb{R}_+) \Big\} \subset L^2(\Omega). \label{Hforiequal1}
\end{equation}
Using It$\hat{o}$'s formula on X for $X \in L^2(\Omega, \mathcal{F}, \mathbb{P})$, we can prove that $\mathcal{H}$ is dense in  $L^2(\Omega, \mathcal{F}, \mathbb{P})$:

\begin{eqnarray*}
(X_{t+t_0} - X_{t_0}) &=& \int^t_0 d(X_{s+t_0} - X_{t_0}) \\
&& + \sum_{0 <s \leq t} \big[ (X_{s+t_0} -X_{t_0}) - (X_{(s+t_0)-} -X_{t_0}) -\Delta X_{s+t_0} \big] \mbox{ \cite{Nualart}}\\
&=& \int^t_0 d(X_{s+t_0} - X_{t_0}). \\
\mbox{ Set } t_0 = 0 &\mbox{ then, }& \\
X_t &=& \int^t_0 dX_s = \int^t_0 dX^1_s = \int^t_0 dY^1_s + m_1t = \mathbb{E}(X_t) + \int^t_0 dY^1_s.\\
\end{eqnarray*}
We can derive from the above a simplified, 1-dimensional form of the {\bf Predictable Representation Property (PRP)}:

\[ \forall X \in L^2(\Omega, \mathcal{F}, \mathbb{P}), \forall t \in [0,T], \exists \phi \in L^2(\Omega) \mbox{ s.a. }\]
\[ X_t =  \mathbb{E}(X_T) + \int^t_0 \phi_s dY^1_s. \]
As a consequence, we also have the following theorem

\begin{thm}
$\mathcal{H}$ is dense in $L^2(\Omega, \mathcal{F}, \mathbb{P})$. In other words,
$ L^2(\Omega, \mathcal{F}, \mathbb{P}) = \mathbb{R} \bigoplus \mathcal{H}$.
\label{onedimPRP}
\end{thm}

We want now to take the result of theorem \ref{onedimPRP} further to a multi-dimensional adaptation. As in the sections above, we wish to find an integral representation for martingales in $L^2(\Omega, \mathcal{F}, \mathbb{P})$. But this time in order to do so, we need to find a spamming set of $\mathcal{M}^2$. This is achievable by constructing a set of pairwise orthogonal martingales $\{ H^i \mid i \geq 1\}$ where:

\begin{eqnarray}
H^i = Y^i + \sum^i_1 a_{i,i-j}Y^{i-j}. 
\label{polynomial1}
\end{eqnarray}

Here, orthogonality is understood as in definition \ref{strongorthog}.

How can we find the coefficients $a_{i,j}$ in equation \ref{polynomial1} to form an orthogonal polynomial set? \\
Consider:
\begin{eqnarray*}
S_1 &=& \big\{ \mbox{polynomials in } \mathbb{R} \mid \langle P(x),Q(x) \rangle = \int^{+ \infty}_{- \infty} P(x)Q(x)x^2\nu(dx) + \sigma^2 P(0)Q(0) \big\} \mbox{ and} \\
S_2 &=& \big\{ \sum^n_i a_i Y^i \mid n\in \mathbb{N}, a_i \in \mathbb{R}, \langle Y^i, Y^j \rangle = m_{i+j} + \sigma^2 1_{\{i=j=1\}} \big\} \supseteq \{ H^i \mid i \geq 1\}.
\end{eqnarray*}

Since there is a clear isometry between $S_1$ and $S_2$ , namely $x^{i-1} \longleftrightarrow Y^i$ \cite{Nualart}, we can create an orthogonal basis for $\mathcal{M}^2$ by orthogonalising $\{1, x, x^2,x^3....\}$. To this purpose, some well-known polynomials can be used, such as the Laguerre or the Hermite polynomials:
\[ H_n(x) = (-1)^n e^{\frac{x^2}{2}}\frac{d^n}{dx^n}e^{\frac{-x^2}{2}}. \]
\\
We take an interest in the following spaces:
\begin{eqnarray*}
\mathcal{H}^{(i_1,...,i_j)} &=& \Big\{ \int^{\infty}_0 \int^{t_1-}_0 ... \int^{t_{j-1}-}_0 f(t_1,...,t_j)dH^{i_j}_{t_j}...dH^{i_1}_{t_1}  \mid  f \in L^2(\mathbb{R}^j_+) \Big\} \subset L^2(\Omega). \\
\end{eqnarray*}
It is known that if $(i_1,...,i_j) \neq (j_1,...,j_l)$ then $ \mathcal{H}^{(i_1,...,i_j)} \perp \mathcal{H}^{(j_1,...,j_l)}$. \cite{Nualart} \\
These space are basically a multi-dimensional adaptation of $\mathcal{H}$ defined earlier. Using these, we can prove an extension of theorem \ref{onedimPRP} 
\begin{thm}
 $L^2(\Omega, \mathcal{F}) = \mathbb{R} \oplus \big( \bigoplus^{\infty}_{j=1} \bigoplus_{i_1...i_j \geq 1}  \mathcal{H}^{(i_1,...,i_j)} \big). \label{theorem1}$ 
\end{thm}

\begin{proof} To show this we start by noting that:
\begin{eqnarray*}
\mathcal{P} = \big\{ X^{k_1}_{t_1} \prod^n_2 (X_{t_i}-X_{t_{i-1}})^{k_i} \mid n \geq 0, k_i \geq 1 \big\}
\end{eqnarray*}
is dense in $L^2(\Omega, \mathcal{F})$. Indeed, take $ Z \in L^2(\Omega, \mathcal{F})$, $Z \perp \mathcal{P}$. For some finite set $\{0 <s_1<...<s_m\}$ there is $Z_{\epsilon} \in L^2(\Omega, \omega(X_{s_1},...,X_{s_m}))$ s.a.: 
\[
\mathbb{E}\big[ (Z-Z_{\epsilon})^2\big] < \epsilon.
\]
There exists a Borel function where: $ Z_{\epsilon} = f_{\epsilon}(X_{s_1},X_{s_2}-X_{s_1},...,X_{s_m}-X_{s_{m-1}})$, which can be approximated by polynomials. Additionally, $\mathbb{E}[ZZ_{\epsilon}] = 0$. Then:

\begin{eqnarray*}
 \mathbb{E}[Z^2] = \mathbb{E}[Z(Z-Z_{\epsilon})] \leq \sqrt{ \mathbb{E}[Z^2]\mathbb{E}[(Z-Z_{\epsilon})^2] } \leq \sqrt{ \epsilon\mathbb{E}[Z^2],} 
\end{eqnarray*}
so Z = 0 a.s. as $ \epsilon \rightarrow 0$.
If we can represent the terms of $\mathcal{P}$ we can thus represent $ L^2(\Omega, \mathcal{F})$. For this, we rely on the following lemma:

\begin{eqnarray}
(X_{t_0+t} - X_{t_0})^k  &=& \mathbb{E}[(X_{t_0+t} - X_{t_0})^k] \nonumber \\
&&+ \sum^k_{j=1} \sum_{(i_1,...,i_j)} \int^{t_0+t}_{t_0} \int^{t_1-}_{t_0}...\int^{t_{j-1}-}_{t_0} h^k_{(i_1,...,i_j)}(t,t_0,...,t_j)dH^{i_j}_{t_j}...dH^{i_1}_{t_1}
\label{lemma1}
\end{eqnarray}

where $h^k_{(i_1,...,i_j)} \in L^2(\mathbb{R}^j_+)$. \cite{Nualart}\\
For any $0 \leq t < s \leq u <v$, k,l $\geq$1, we have that $(X_s-X_t)^k(X_v-X_u)^l = AB$ where A and B are of the form stated in lemma \ref{lemma1}.Then:

\begin{eqnarray*}
AB = \int^{\infty}_0 \int^{u_1-}_0 ... \int^{u_m-}_0  \int^{t_1-}_0 ... \int^{t_{n-1}-}_0 \prod_{i=u,u_1...t_n} 1(i) \\
\times h^l_{(j_1,...,j_m)}(v,u,u_1,...,u_m)h^k_{(i_1,...,i_n)}(s,t,t_1,...,t_n) \\
\times dH^{i_n}_{t_n}...dH^{i_1}_{t_1}dH^{j_m}_{u_m}...dH^{j_1}_{u_1}.
\end{eqnarray*}

This then gives the {\bf Chaotic representation property (CRP)}: $\forall F \in L^2(\Omega, \mathcal{F})$,
\begin{equation}
F = \mathbb{E}[F] + \sum^{\infty}_j \sum_{i_1,..,i_j \geq1} \int^{\infty}_0 \int^{t_1-}_{t_0}...\int^{t_{j-1}-}_{t_0} f_{(i_1,...,i_j)}(t_1,...,t_j)dH^{i_j}_{t_j}...dH^{i_1}_{t_1}
\label{CRP}
\end{equation}

where $ f_{(i_1,...,i_j)} \in L^2(\mathbb{R}^j_+)$.
Theorem 2 is then a consequence of the CRP \ref{CRP}.   
\end{proof}

To attain a representation result for martingales, we observe that the CRP (equation \ref{CRP}) can be transformed in the following way: $\forall F \in L^2(\Omega, \mathcal{F})$,

\begin{eqnarray*}
F - \mathbb{E}(F) &=& \sum^{\infty}_j \sum_{i_1,..,i_j \geq1} \int^{\infty}_0 \int^{t_1-}_{t_0}...\int^{t_{j-1}-}_{t_0} f_{(i_1,...,i_j)}(t_1,...,t_j)dH^{i_j}_{t_j}...dH^{i_1}_{t_1} \\
&=& \sum^{\infty}_{i_1=1} \int^{\infty}_0 f_{i_1}(t_1) dH^{i_1}_{t_1}  +  \sum^{\infty}_{i_1=1} \int^{\infty}_0 \Big[ \sum^{\infty}_{j=2} \sum_{i_2,..,i_j \geq1} \int^{t_1-}_0 \cdots \\
&& \int^{t_{j-1}-}_0 f_{(i_1,...,i_j)}(t_1,...,t_j) dH^{i_j}_{t_j}...dH^{i_2}_{t_2} \Big] dH^{i_1}_{t_1} \\
&=&  \sum^{\infty}_{i_1=1} \int^{\infty}_0 \Big[ f_{i_1}(t_1) + \sum^{\infty}_{j=2} \sum_{i_2,..,i_j \geq1} \int^{t_1-}_0 \cdots  \int^{t_{j-1}-}_0 f_{(i_1,...,i_j)}(t_1,...,t_j) dH^{i_j}_{t_j}...dH^{i_2}_{t_2} \Big] dH^{i_1}_{t_1} \\
&=& \sum^{\infty}_{i=1} \int^{\infty}_0 \phi^i_{t_1}dH^i_{t_1},
\end{eqnarray*}

where $\forall i $
\[ \phi^i_{t_1} =  f_{i_1}(t_1) + \sum^{\infty}_{j=2} \sum_{i_2,..,i_j \geq1} \int^{t_1-}_0 \cdots  \int^{t_{j-1}-}_0 f_{(i_1,...,i_j)}(t_1,...,t_j) dH^{i_j}_{t_j}...dH^{i_2}_{t_2} \]
and $\phi^i_t$ is predictable. Hence the following result:

\begin{defn}[{\bf Predictable Representation property (PRP)}]
 $\forall F \in L^2(\Omega, \mathcal{F})$, there is a $\phi^i_t$ predictable such as:

\begin{equation} F = \mathbb{E}(F) + \sum^{\infty}_{i=1} \int^{\infty}_0 \phi^i_s dH^i_s. \label{PRP} \end{equation}

Then $\forall M \in \mathcal{M}^2$ with $M_{\infty} \in L^2(\Omega,\mathcal{F})$ and $M_t = \mathbb{E}(M_{\infty} \mid \mathcal{F}_t) $ $\forall t$, The PRP (equation \ref{PRP}) gives us:
\[ M_t = \sum^{\infty}_{i=1} \int^{t}_0 \phi^i_s dH^i_s. \]

\end{defn}

While this result is very similar to the conventional Brownian-motion based MRT, it presents the advantage of being more general and reaching out to wider sets and more elaborate martingales.  Indeed, as we can see through the following:
\begin{thm}[{\bf L$\acute{e}$vy-Ito theorem}]
 $ \forall (X_t)$ L$\acute{e}$vy processes on $(\Omega, \mathcal{F}, \mathcal{F}_t, \mathbb{P})$ - where the distribution of $X_1$ is parametrized by $(\beta, \sigma^2, \upsilon)$ in the L$e$vy-Khintchine theorem - X decomposes as follow \cite{Applebaum} :

\begin{equation}
X_t = \beta t+ \sigma W_t + J_t + M_t,
\label{levyIto}
\end{equation}

where
$\beta \in \mathbb{R}$, $\sigma^2 \geq 0$ and $\upsilon$ is a measure on $\mathbb{R}/\{0\}$ such as $\int_{\mathbb{R}/\{0\}} 1 \wedge x^2 \upsilon(dx) < \infty$, and
\begin{itemize}
  \item $W_t$ is a Brownian Motion,
  \item $\Delta X_t = X_t - X_{t-}$ for t $\geq$ 0 is an indep. Poisson point process with intensity $\upsilon$,
  \item $ J_t = \sum_{s \leq t} \Delta X_s 1_{\{ |\Delta X_s| > 1\}} $, and
  \item $M_t$ is a martingale with jumps: $\Delta M_t = \Delta X_t 1_{\{ |\Delta X_t| > 1\}}$.
\end{itemize} 
\end{thm}

The result \ref{levyIto} can be re-written as done in \cite{Lokka}:

\begin{equation}
L_t = \beta t + \sigma W_t + \int^t_0 \int_{\mathbb{R}_0} z (\mu - \upsilon)(t,dz)
\label{levyito-lokka}
\end{equation}
 where $\mu$ is a Poisson random measure, $\mathbb{R}_0 = \mathbb{R} / \{0\}$, and $\beta, \sigma$ are defined as in the theorem above. 

\begin{thm}[{\bf lemma 12} \cite{Lokka}]
$\exists \{ \Lambda_n \}^{\infty}_{n=1}$ partitioning $\mathbb{R}_0$ and $z_n \in \Lambda_n \subseteq \mathbb{R}$ such as
\[ \int_{\mathbb{R}_0} z(\mu-\upsilon)(t,dz) = \sum^{\infty}_{n=1} z_n(\mu-\upsilon)(t,\Lambda_n) \]
where the processes $(\mu-\upsilon)(t,\Lambda_n)$ are all compensated Poisson processes with intensity $\upsilon(\Lambda_n)$. Hence, 
\begin{equation}
\int^t_0 \int_{\mathbb{R}_0} z (\mu - \upsilon)(t,dz) = \int^t_0 \int_{\mathbb{R}_0} z \tilde{N}(ds,dz).
\label{lemma-lokka}
\end{equation}
\end{thm}

Then we can apply the PRP, result obtained in equation \ref{PRP}: $\forall m \in \mathcal{M}^2$,
\begin{eqnarray*}
m_t &=& \sum^{\infty}_{i=1} \int^{t}_0 \phi^i_s dH^i_s \\
&=& \sum^{\infty}_{i=1} \int^{t}_0 \phi^i_s (\sigma dW^i_s + d( \int^s_0 \int_{\mathbb{R}_0} z (\mu - \upsilon)^i(s,dz))) \mbox{ , since $H^i$ is a l$\acute{e}$vy process result \ref{levyito-lokka} applies,} \\
&=& \sum^{\infty}_{i=1} \int^{t}_0 \hat{\phi}^i_s  dW^i_s +  \int^t_0 \psi^i(s,z) (\mu - \upsilon)^i(s,dz) \\
&=& \sum^{\infty}_{i=1} \int^{t}_0 \hat{\phi}^i_s  dW^i_s +  \int^t_0 \psi^i(s,z) \tilde{N}^i(ds,dz) \mbox{- using result \ref{lemma-lokka},}
\end{eqnarray*}
which is similar to a result developed in \cite{Lokka}:\\
$\forall m \in \mathcal{M}^2$, and where ($m_t$) is of dimension n, $\exists \hat{\phi}(t), \psi(t,z)$ predictable n-dimensional processes such as
\begin{equation}
m_t = \int^{t}_0 \hat{\phi}_s  dW_s +  \int^t_0 \psi(s,z) \tilde{N}(ds,dz),
\label{wiener-poissonmrt}
\end{equation}
 where $W_t$, $\tilde{N}(t,z) $ are also n-dimensional and 
\begin{eqnarray*}
\mathbb{E}(\int^T_0 \hat{\phi}^2_s ds) < \infty ,\\
\mathbb{E}(  \int^T_0 \psi(s,z)^2  \upsilon(dz) ds) < \infty.
\end{eqnarray*}

\section{The Clark-Ocone formula and explicit representation of the integrand}

We have explored in the previous chapter the essentials of the Martingale Representation Theorem in its various forms, notably within the continuous space driven by the Brownian motion and beyond with L$\acute{e}$vy and jump processes. It is generally agreed that such a representation formula exists in $\mathcal{M}^2$ because of the Hilbert structure of the probability spaces these processes live in. More specifically, in each of these $L^2(\Omega, \mathcal{F}, (\mathcal{F})_{t \geq 1})$,there is a set of integrals of fundamental processes $\mathcal{H}$ such as $\mathcal{H}$ is dense in  $\mathcal{M}^2 \subseteq L^2$. and:

\[ \forall X_t \in \mathcal{M}^2, X_t = \int^t_0 \phi_s dq_s \mbox{, } q_t \in \mathcal{H}. \]

While this gives a general formula to martingales, it gives no indication as to what $\phi_t$ is. However knowing the form of the integrand $\phi_t$  is of central importance in the applications of the MRT, notably in the representation of portfolio dynamics and trading strategy optimization.\\
\\
To explore this question, we first introduce basics of Malliavin calculus relevant to the representation of the integrand $\phi_t$. We then cover an important result providing a formula for the integrand, the Clark-Ocone Formula.

\subsection{Malliavin Calculus}

Malliavin calculus is the extension of the calculus of variations from functions to stochastic processes over a finite or infinite dimensional space. To develop on this topic, we follow the work of Li (2011) \cite{Li} and $\O$ksendal (1997) \cite{Oksendal2}. \\
We work on $L^2(\Omega, \mathcal{F}_t, \mathbb{P})$, where $\mathcal{F}_t$ is the $\sigma$-algebra generated by a Brownian Motion $W_t$, $t \in [0,T]$, and $\Omega = C_0([0,T])$. Here, $\mathcal{F}= \{\mathcal{F}_t \mid t \in [0,T]\}$ is the initial filtration augmented by $\mathbb{P}$-zero measure sets. \\

\begin{defn}
a function $g: [0,T]^n \rightarrow \mathbb{R}$ is symmetric if
$ g(t_{\sigma_1},...,t_{\sigma_n}) =  g(t_1,...,t_n)$
for any permutation $\sigma=(\sigma_1,...,\sigma_n)$ of (1,..,n).
\end{defn}

Let:
\[ \tilde{L}^2([0,T]^n) = \{  g: [0,T]^n \rightarrow \mathbb{R} \mid \mbox{symmetric square integrable functions}\}  \subset L^2([0,T]^n) \mbox{, and} \]
\[ S_n = \{(t_1,..,t_n) \in [0,T]^n \mid t_i \leq t_j, \forall i \leq j \}. \]

\begin{defn}[{\bf the n-fold interated It$\hat{o}$ integral:}]
We define the n-fold iterated It$\hat{o}$ integral where f is a deterministic function defined on $S_n$ as:
\[ J_n(f) = \int^t_0\int^{t_n}_0 ... \int^{t_2}_0 f(t_1,...,t_n) dW_{t_1}...dW_{t_n}, \]
and note that $J_n(f) \in L^2(\Omega)$.
\end{defn}

\begin{defn}
for $g \in \tilde{L}^2([0,T]^n)$, set
\[I_n(g) = \int_{[0,T]^n} g(t_1,...,t_n) dW_{t_1}...dW_{t_n}=n!J_n(g). \]
\label{defInJn}
\end{defn}

\begin{defn}[{\bf n-th Wiener Chaos} \cite{Peccati}]
The n-th Wiener Chaos $C_n$ is defined as:
\[ C_n = \{ I_n(f) \mid f \in L^2(\mathbb{P}) \} \mbox{, } n \geq 1.\]
\end{defn}

The operators $I_n$ and $J_n$ have a couple of useful properties, notably the following:
\begin{thm}
 $\forall f_n \in \tilde{L}^2([0,T]^n)$ with G being a borel set $\subseteq [0,T]$,
\begin{equation}
\mathbb{E}(I_n(f_n)) \mid \mathcal{F}_G) = I_n(f_n \prod^n_{i=1} 1_{G}(t_i)),
\label{condiexpI}
\end{equation}
where $\mathcal{F}_G$ is a completed $\sigma$-field:
\[  \mathcal{F}_G = \sigma \{ \int^T_0 1_{A}(t)dW_t \mid A- \mbox{borel sets}, A \subseteq G \}. \]
\end{thm}

\begin{thm}[{\bf Wiener-It$\hat{o}$ Chaos Expansion \cite{Oksendal2}}]
$F \in L^2(\Omega)$. Then there is a unique sequence $(f_n)_n$ of deterministic functions $f_n \in \tilde{L}^2([0,T]^n)$ such as:
\begin{equation}
F = \sum^{\infty}_{n=0} I_n(f_n) = \mathbb{E}(F) + \sum^{\infty}_{n=1} I_n(f_n),
\label{wienerchaosexp}
\end{equation}
and 
\[ ||F||^2_{L^2(\Omega)} = \sum^{\infty}_{n=0} n! ||f_n||^2_{L^2([0,T]^n)} < \infty.\]
\end{thm}

Moreover, we have:
\[ J_n(F) = I_n(f_n), \]
so J can be seen as an orthogonal projection of F on the n-th Chaos $C_n$ \cite{Li}:
\[ F = \sum^{\infty}_{n=0} J_n(F).  \]

At this stage, we denote the following space:

\begin{defn}
\[ \mathcal{P} = \{ F: \Omega \rightarrow \mathbb{R} \mid F(\omega) = p (W_{t_1}...W_{t_n}), p(x) \mbox{: polynomial, }  p \in C_0[0,T] \} \]
and also we introduce the Cameron-Martin space $\mathcal{H}$:
\[ \mathcal{H} = \big\{ \gamma: [0,T] \rightarrow \mathbb{R} \mid \gamma(t) = \int^t_0 \dot{\gamma}(s)ds, |\gamma|^2_{\mathcal{H}} = \int^T_0 \dot{\gamma}(s)^2ds <\infty \big\} \subseteq C_0[0,T]\]
\end{defn}

At this stage, we can now introduce the concept of directional derivative;

\begin{defn}
$ \forall \in \mathcal{P}$ the directional derivative $D_{\gamma}F(\omega)$ $\forall \gamma \in \mathcal{H}$ is defined as:

\begin{equation}
D_{\gamma}F(\omega) = \lim_{\epsilon \rightarrow 0 }\frac{F(\omega +\epsilon \gamma) - F(\omega)}{\epsilon},
\label{dirder1}
\end{equation}

or alternatively,
\begin{equation}
D_{\gamma}F(\omega) = \sum^n_{i=1} \frac{\partial p}{\partial x_i}(W_{t_1},...,W_{t_n}) \gamma(t_i).
\label{dirder2}
\end{equation}

$\forall F \in \mathcal{P}$, the function $D_{\gamma}: \mathcal{H} \rightarrow L^2(\Omega )$ is continuous \cite{Oksendal2} and obeys to the product rule:
\[ D_{\gamma}(FG) = F D_{\gamma}G + G D_{\gamma}F. \]
\end{defn}

\begin{thm}[{\bf Riesz Representation theorem} \cite{Rudin}]
For every F that has a defined derivative $D_{\gamma}F$ $\forall \gamma \in \mathcal{H}$, $\exists ! \bigtriangledown F(\omega) \in \mathcal{H}$ such as:
\begin{equation}
D_{\gamma}F(\omega) = \langle \gamma, \bigtriangledown F(\omega) \rangle_{\mathcal{H}} = \int^T_0 \dot{\bigtriangledown F(\omega)} \dot{\gamma} dt.
\label{riesz}
\end{equation}
\end{thm}

The Malliavin derivative $D_tF$ of F can then de defined as such:

\begin{defn}[{\bf Malliavin derivative}]
 $D_t: \mathcal{P} \rightarrow L^2([0,T] \times \Omega)$  is the Radon-Nikodym derivative of $\bigtriangledown F(\omega)$:

\begin{equation}
D_{\gamma}F(\omega) = \int^T_0 D_tF(\omega) \dot{\gamma}(t) dt. \nonumber
\end{equation}

The Malliavin derivative $D_t F$  is also continuous, closable and obeys to the product rule. 
\label{malliaderdef}
\end{defn}

Then we introduce the following semi-norm on $\mathcal{P}$ :

\begin{equation}
\| |F| \|_{1,2} = \Big[ \mathbb{E}(|F|^2)+ \mathbb{E}\big(\| D_tF\|^2_{L^2([0,T])}\big) \Big]^{1/2}.
\label{d12norm}
\end{equation}

The completion of  $\mathcal{P}$ under this new norm in \ref{d12norm} creates a Banach space $\mathbb{D}_{1,2}$ called a Sobolev space. $\mathbb{D}_{1,2}$  is a Hilbert space such as \cite{Oksendal2}:
\[ \mathbb{D}_{1,2}= \Big\{ F \in L^2(\Omega) \mid \{F_n\}_{n \in \mathbb{N}} \rightarrow F, \{D_tF_n\}_{n \in \mathbb{N}} \mbox{ Cauchy in } L^2([0,T] \times \Omega) \Big\}. \]

\begin{thm}
If $\phi : \mathbb{R}^n \rightarrow \mathbb{R}$ is Lipschitz i.e. $\forall$ x,y $ \in \mathbb{R}^n$ and $\exists$ K constant, we have:
\[ | \phi(x) - \phi(y)| \leq K |x-y|, \]
and F = ($ F_i$) $ F_i\in \mathbb{D}_{1,2}$ $\forall i \leq n$,
Then $\phi(F) \in  \mathbb{D}_{1,2}$ and 
\[ D_t \phi(F) = \sum^n_{i=1} \frac{\partial \phi}{\partial x_i}(F) D_t F_i \mbox{ \cite{Li}}.\]
\end{thm}

The Malliavin derivative $D_t$ aslo has the following useful property:
\begin{thm}[Oksendal \cite{Oksendal2}]
For any $F(\omega) = I_n(f_n)$ where $f_n \in \tilde{L}^2([0,T]^n)$,  F $\in \mathbb{D}_{1,2}$ and
\begin{equation}
D_t F(\omega) = n I_{n-1}(f_n).
\label{diffIn}
\end{equation}
\end{thm}

To match the Malliavin derivative, there is a special form of integration similar to stochastic integration:

\begin{defn}[{\bf Skorohod integral of u}]

$u(t,\omega)$ is a $\mathcal{F}_T$-measurable r.v. for all t $\in [0,T]$ such as:
\[ \mathbb{E}(u^2(t)) < \infty, \]
and u(t) has a Wiener-It$\hat{o}$ chaos expansion $u(t) = \sum^{\infty}_{n=0} I_n(f_n) $ where $f_n \in \tilde{L}^2([0,T]^n)$. set
\[ \tilde{f}_n(t_1,...,t_{n+1}) = \frac{1}{n+1} \big[ f_n(t_1,...,t_{n+1})+ f_n(t_2,...,t_{n+1},t_1)+... \big]. \]

For u(t) = $\sum^{\infty}_{n=0} I_n(f_n)$, the Skorohod integral is defined as:
\begin{equation}
\delta(u) = \int^t_0 u(t)\delta W_t = \sum^{\infty}_{n=0} I_{n+1}(\tilde{f}_n)
\label{skorohod}
\end{equation}
whenever :

\begin{equation}
\mathbb{E}(\delta(u)^2) =  \sum^{\infty}_{n=0} (n+1)! \|\tilde{f}_n\|^2_{L^2} < \infty,
\label{skorohodcond}
\end{equation}

in which case, we say $u \in Dom(\delta)$.
\end{defn}

The Skorohod integration $\delta$ and the Malliavin derivative $D_t$ are connected through the following version of the fundamental theorem of calculus

\begin{thm}[{\bf The fundamental theorem of calculus}]
Let u(s) be a Skorohod-integrable stochastic process contained in $\mathbb{D}_{1,2}$ and that $\forall t \in [0,T]$, $D_tu$ is also Skorohod-integrable. Then:
\[ D_t \Big( \int^T_0 u(s) \delta W_s \big) = \int^T_0 D_t u(s) \delta W_s + u(t) \]
\end{thm}

\begin{thm}
$u \in Dom(\delta) \Longrightarrow \delta(u) \in L^2 $
\end{thm}

The Skorohod integral has a couple of nice properties,  notably that
\[ \mathbb{E}( \delta(u)) = 0\]
as it is an iterated integral of Brownian motion, and hence has zero expectation.

\begin{thm}
u(t) is an $\mathbb{F}$-adpated r.v. such as 
\[ \mathbb{E}( \int^T_0 u^2(t)dt) < \infty. \]
Then u $\in Dom(\delta)$ and the Skorohod integral coincides with the It$\hat{o}$ integral:
\[
\int^T_0 u(t) \delta W_t = \int^T_0 u(t) dW_t.
\]
\label{skoequality}
\end{thm}
This theorem illustrates the usefulness of the Skorohod integral: it is an equivalent of the regular stochastic integral, but is applicable to stochastic processes that are $\mathbb{F}$-adapted or not. The following theorem describes easily in which case we are in.

\begin{thm}
u(t) is $\mathcal{F}_T$-measurable and $\mathbb{E}(u^2(t)) < \infty$. We have:
\[ u(t) = \sum^{\infty}_{n=0} I_n(f_n). \]
u(t) is $\mathbb{F}$-adapted if and only if \cite{Oksendal3}
\begin{equation}
f_n(t_1,...,t_n,t) 1_{t < \max_{i \leq n} t_i} = 0.
\end{equation}
\end{thm}

Theorem \ref{skoequality} enables us to establish the following equality:
\begin{equation}
\int^T_0 \sum^{\infty}_n J_n(f_n) dW_t = \sum^{\infty}_n J_{n+1}(f_n).
\label{intofJ}
\end{equation}

\subsection{The Clark-Ocone formula: Result and application to the MRT}

\begin{thm}[{\bf The Clark-Ocone formula} \cite{Oksendal2}]
$ \forall F \in \mathbb{D}_{1,2}$ where F is $\mathcal{F}_T$-measurable, the following representation holds:

\begin{equation}
F(\omega) = \mathbb{E}(F) + \int^T_0 \mathbb{E}(D_t F \mid \mathcal{F}_t) dW(t).
\label{clarkocone}
\end{equation}
\end{thm}

\begin{proof} This result is a statement on the It$\hat{o}$ representation theorem seen in chapter 1 result \ref{itorepth}; $\forall F \in L^2(\Omega)$ where F is  $\mathcal{F}_T$-measurable, $\exists !  \phi(t)$ such as:
\[ F = \mathbb{E}(F) + \int^T_0 \phi(t) dW_t. \]
The difference is that here, we have an explicit form for $\phi(t)$ and we claim $\phi(t) = \mathbb{E}(D_t F \mid \mathcal{F}_t) $.\\
We know from the Wiener-It$\hat{o}$ chaos expansion (result \ref{wienerchaosexp}) that:
\[ F = \sum^{\infty}_{n=0} I_n(f_n) \]
where $f_n \in \tilde{L}^2([0,T]^n)$. We then have the following:
\begin{eqnarray*}
\int^T_0 \mathbb{E}(D_t F \mid \mathcal{F}_t) dW_t &=& \int^T_0 \mathbb{E}(D_t \sum^{\infty} I_n(f_n) \mid \mathcal{F}_t) dW_t \mbox{ ,using result \ref{wienerchaosexp} } \\
&=& \sum^{\infty} \int^T_0 \mathbb{E}(D_t I_n(f_n) \mid \mathcal{F}_t) dW_t \\
&=& \sum^{\infty}_{n=1} \int^T_0 \mathbb{E}(n I_{n-1}(f_n) \mid \mathcal{F}_t) dW_t \mbox{ ,using \ref{diffIn}} \\
&=& \sum^{\infty}_{n=1} n \int^T_0 \mathbb{E}(I_{n-1}(f_n) \mid \mathcal{F}_t) dW_t \\
&=& \sum^{\infty}_{n=1} n \int^T_0 I_{n-1}(f_n \prod^{n-1}_i 1_{\mathcal{F}_t}(t_i))  dW_t \mbox{ ,using \ref{condiexpI}} \\
&=& \sum^{\infty}_{n=1} n \int^T_0 I_{n-1}(f_n \prod^{n-1}_i 1_{[0,t]}(t_i))  dW_t \\
&=& \sum^{\infty}_{n=1} n(n-1)! \int^T_0 J_{n-1}(f_n \prod^{n-1}_i 1_{[0,t]}(t_i))  dW_t \mbox{ ,using definition \ref{defInJn}} \\
&=& \sum^{\infty}_{n=1} n! \int^T_0 J_{n-1}(f_n)  dW_t \\
&=& \sum^{\infty}_{n=1} n!  J_{n}(f_n)  \mbox{ ,using result \ref{intofJ}} \\
&=&  \sum^{\infty}_{n=1} I_{n}(f_n) = \sum^{\infty}_{n=0} I_{n}(f_n) - I_{0}(f_0) = F - \mathbb{E}(F).
\end{eqnarray*}
\end{proof}

Looking at the martingale representation theorem in the continuous filtration as seen in section 1.2.1, we see that for all $(m_t) \in \mathcal{M}^2 \cap \mathbb{D}_{1,2}$,
\[ m_t = \mathbb{E}( m_{\infty} \mid \mathcal{F}_t ) = \int^t_0 \phi_s dW_s \]
for some $\phi_t$ deterministic where, through result \ref{clarkocone}, we know that $\phi_t = \mathbb{E}(D_t m_{\infty} \mid \mathcal{F}_t) $. So the can re-write the martingale representation theorem as such:

\begin{eqnarray}
 \forall m \in \mathcal{M}^2 \cap \mathbb{D}_{1,2}, m_t &=& \mathbb{E}(m_{\infty} \mid \mathcal{F}_t) \nonumber \\
&=& \mathbb{E}(\int^{\infty}_0 \mathbb{E}( D_s m_{\infty} \mid \mathcal{F}_s) dW_s \mid \mathcal{F}_t) \nonumber \\
&=& \int^t_0 \mathbb{E}( D_s m_{\infty} \mid \mathcal{F}_s) dW_s.
\label{contMRTmod}
 \end{eqnarray}

Now, for this representation result to hold, it is important to know whether $ \int^t_0 \mathbb{E}( D_s m_{\infty} \mid \mathcal{F}_s)^2 ds < \infty$ for all t $\in [0,T]$.\\
Since $m_t \in \mathcal{M}^2 \cap \mathbb{D}_{1,2}$ for any t in [0,T], we know that:
\[ \| |m_{\infty}| \|_{1,2}= \Big[ \mathbb{E}(|m_{\infty}|^2)+ \mathbb{E}\big(\| D_t m_{\infty}\|^2_{L^2([0,T])}\big) \Big]^{1/2} < \infty,\]
hence
\[ \mathbb{E}\big(\| D_t m_{\infty}\|^2_{L^2([0,T])}\big)  < \infty, \]
which implies that indeed:
\begin{eqnarray*}
\mathbb{E}(\int_{[0,T]} \mathbb{E}( D_s m_{\infty} \mid \mathcal{F}_s)^2 ds) &<&  \mathbb{E}(\int_{[0,T]} \mathbb{E}( (D_s m_{\infty})^2 \mid \mathcal{F}_s) ds) \\
&=& \mathbb{E}(\mathbb{E}(\int_{[0,T]}  (D_s m_{\infty})^2  ds\mid \mathcal{F}_s) ) \\
&=& \mathbb{E}(\int_{[0,T]}  (D_s m_{\infty})^2  ds ) \\
&=& \mathbb{E}( \| D_s m_{\infty} \| ^2_{L^2}) < \infty.
\end{eqnarray*}

So $\int_{[0,T]} \mathbb{E}( D_s m_{\infty} \mid \mathcal{F}_s)^2 ds < \infty$ a.s. and thus the martingale representation theorem in result \ref{contMRTmod} is well-defined.

\subsection{Explicit integrand representation beyond continuous processes: Poisson Malliavin Calculus and Clark-Ocone formula applied to jump processes}

\subsubsection{Poisson Malliavin Calculus}

In this setting we work with the compensated poisson process $\tilde{N} (t,z)$, a martingale contained in $ L^2( [0,T] \times \mathbb{R}_0^n ) $. It evolves on a complete probability space $(\Omega, \mathcal{F}, P)$, where $\mathcal{F}_t$  is the $\sigma$-algebra generated by $\tilde{N} (s,z)$, $ 0 \leq s \leq t$.\cite{Oksendal3} 
\\
Let $\mu$ be a l$\acute{e}$vy measure, and $\lambda$ the regular Lebesgue measure. $\tilde{L}^2 ( \lambda \times \mu )$ denote the space of symmetric functions in $L^2 ( \lambda \times \mu )$, which itself is the space of suitably square integrable functions:
\begin{eqnarray*}
\| f \|^2_{ L^2 ( \lambda \times \mu ) } = \int_{ ([0,T] \times \mathbb{R}_0)^n } f^2 dt_1 \mu (dz_1)...dt_n \mu (dz_n) < \infty.
\end{eqnarray*}
Set:
\begin{eqnarray*}
G_n = \Big\{ (t_i, z_i)_{i=1,n} \Big| 0 \leq t_1 \leq ... \leq t_n \leq T \Big| z_i \in \mathbb{R}_0 \Big\}
\end{eqnarray*}
 and  $ L^2 (G_n) = \big\{ g \in G_n \big| \| g \|^2_{ L^2 ( G_n) } < \infty \big\}.$
\\
As earlier, the n-fold Iterated Stochastic integrals are defined as bellow:
\begin{eqnarray*}
J_n(g) &=& \int^T_0 \int_{\mathbb{R}_0} ... \int^{t_2}_0 \int_{\mathbb{R}_0} g(t_1,z_1...) \tilde{N} (dt_1,dz_1) ....\tilde{N} (dt_n,dz_n) \mbox{ , } \forall g \in  L^2 (G_n),\\
I_n(g) &=& n! J_n(g) = \int_{([0,T] \times \mathbb{R}_0)^n} g(t_1,z_1...) \tilde{N} (dt_1,dz_1) ....\tilde{N} (dt_n,dz_n),
\end{eqnarray*}
\\
and as in the continuous case, there is a Wiener-Ito Chaos expansion:\\
$ \forall F \in L^2(P)$, where F is $\mathcal{F}_T$-measurable, $ \exists $ unique $ f_n \in \tilde{L}^2 ( (\lambda \times \mu)^n )$ such as: 
\begin{eqnarray*}
F = \sum^{\infty}_{n=0} I_n(f_n).
\end{eqnarray*}

\begin{defn}[{\bf Sobolev Stochastic space}]
$ \mathbb{D}_{1,2} \subset L^2 ( \lambda \times \mu )$ such as \cite{Oksendal3}
\begin{eqnarray*}
\| F \|^2_{  \mathbb{D}_{1,2} } = \sum^{\infty}_{n=1} n n! \| f_n \|^2_{ L^2 ((\lambda \times \mu )^n) } < \infty.
\end{eqnarray*}
\end{defn}

In the Poisson setting, the Malliavin derivative of F has an alternative definition at (t,z) \cite{Oksendal3} :
\begin{eqnarray}
D: \mathbb{D}_{1,2} \rightarrow  L^2 ( P \times \lambda \times \mu) \nonumber \\
D_{t,z} F = \sum^{\infty}_{n=1} n I_{n-1}(f_n(.,t,z)).
\label{poissMallDer}
\end{eqnarray}
In this case as earlier, the D operator follows some classic rules of traditional calculus: \\ \\
{\bf Closability}: \\
$ F, F_k \in \mathbb{D}_{1,2}$ $\forall k \in \mathbb{N}$, and $ F_k \rightarrow F$ in $ L^2(P)$ as well as $ D_{t,z}F_k$ converges. \\
Then $  D_{t,z}F_k \rightarrow D_{t,z} F $.
\\ \\
{\bf Chain Rule:} \\ 
$ F \in \mathbb{D}_{1,2}$ and $\phi$ is continuous on $\mathbb{R}$. Given $\phi(F) \in L^2(P)$ and  $\phi(F+D_{t,z}F) \in L^2(P \times \lambda \times \mu)$ \\
Then $\phi(F) \in \mathbb{D}_{1,2}$  and $D_{t,z}\phi(F) = \phi(F+D_{t,z}F) - \phi(F)$.
\\\\
{\bf Integration by parts:}\\
X(t,z) is Skorohod integrable, $ F \in \mathbb{D}_{1,2}$ and $ X(t,z)(F + D_{t,z}F)$ is also Skorohod integrable, then:

\begin{eqnarray*}
F \int^T_0 \int_{\mathbb{R}_0} X(t,z) \tilde{N} (\delta t,\delta z)&& \\
= \int^T_0 \int_{\mathbb{R}_0} X(t,z)(F + D_{t,z}F) \tilde{N} (\delta t,\delta z) &+&  \int^T_0 \int_{\mathbb{R}_0} X(t,z)D_{t,z}F \mu(dz)dt.
\end{eqnarray*}

\subsubsection{Clark-Ocone formula: a jump diffusion version}

\begin{thm}[{\bf Jump process Clark-ocone formula \cite{Oksendal3}:}]
$\forall F \in \mathbb{D}_{1,2}$ we have:
\begin{equation}
F = \mathbb{E}(F) + \int^T_0 \int_{\mathbb{R}_0} \mathbb{E}(D_{t,z} F \mid \mathcal{F}_t) \tilde{N}(dt,dz)
\label{jumpClarkOcone}
\end{equation}
whenever $ \mathbb{E}(D_{t,z} F \mid \mathcal{F}_t)$ is predictable.
\end{thm}

\begin{proof} The proof is very similar to the one developed in section 2 for result \ref{clarkocone}.Here again, F has a chaos expansion $F= \sum^{\infty}_{n=0} I_n(f_n)$, $f_n \in \tilde{L}^2((\lambda \times \mu)^n)$. Then:

\begin{eqnarray*}
\int^T_0 \int_{\mathbb{R}_0} \mathbb{E}(D_{t,z} F \mid \mathcal{F}_t) \tilde{N}(dt,dz) &=& \int^T_0 \int_{\mathbb{R}_0} \mathbb{E}( \sum^{\infty}_{n=1} nI_{n-1}(f_n(.,t,z)) \mid \mathcal{F}_t) \tilde{N}(dt,dz)  \mbox{, result \ref{poissMallDer}}\\
&=&\sum^{\infty}_{n=1} n \int^T_0 \int_{\mathbb{R}_0} \mathbb{E}((n-1)!J_{n-1}(f_n(.,t,z)) \mid \mathcal{F}_t) \tilde{N}(dt,dz) \\
&=& \sum^{\infty}_{n=1} n! \int^T_0 \int_{\mathbb{R}_0} \mathbb{E}( J_{n-1}(f_n(.,t,z)) \mid \mathcal{F}_t) \tilde{N}(dt,dz) \\
&=& \sum^{\infty}_{n=1} n! \int^T_0 \int_{\mathbb{R}_0} J_{n-1}(f_n(.,t,z)1_{[0,t]})\tilde{N}(dt,dz) \\
&=& \sum^{\infty}_{n=1} n!  J_{n}(f_n(.,t,z)) = \sum^{\infty}_{n=1} I_{n}(f_n(.,t,z)) = F - \mathbb{E}(F).
\end{eqnarray*}

\end{proof}

As seen in Part1 (section 2.3.2), all Poisson pure-jump martingales in $ \mathcal{M}^2$ can be expressed as:
\[ m_t = \int^t_0 \int_{\mathbb{R}_0} g(s,z)\tilde{N}(ds,dz) \]
for suitable integrand g(s,z). Now we know, through result \ref{jumpClarkOcone} that for all $m \in \mathbb{D}_{1,2} \cap \mathcal{M}^2$, we have:
\[ m_{\infty} =  \int^T_0 \int_{\mathbb{R}_0} \mathbb{E}(D_{t,z} m_{\infty} \mid \mathcal{F}_t) \tilde{N}(dt,dz). \]
Hence the can write the Jump-process Martingales representation theorem as follow:
\[ \forall m \in \mathbb{D}_{1,2} \cap \mathcal{M}^2, \mbox{  } m_t = \mathbb{E}(m_{\infty}) =  \int^t_0 \int_{\mathbb{R}_0} \mathbb{E}(D_{s,z} m_{\infty} \mid \mathcal{F}_s) \tilde{N}(ds,dz). \]

Provided $m_{\infty} < \infty$, since  $(m_t) \in \mathbb{D}_{1,2}$ we can see that:

\begin{eqnarray*}
\mathbb{E}( [ \int^t_0 \int_{\mathbb{R}_0} \mathbb{E}(D_{s,z} m_{\infty} \mid \mathcal{F}_s) \tilde{N}(ds,dz)]^2) &=& \mathbb{E}( [ \int^t_0 \int_{\mathbb{R}_0} \mathbb{E}(D_{s,z} m_{\infty} \mid \mathcal{F}_s)^2 \upsilon(ds)dz)]) \\
&<& \mathbb{E}( \| D_t m_{\infty}\|^2_{L^2([0,T] \times \Omega)}) \mbox{ like in section 2} \\
& < &\infty  \mbox{, as } (m_t) \in \mathbb{D}_{1,2}.
\end{eqnarray*}

\subsection{Explicit integrand representation: Clark-Ocone formula applied to general L$\acute{e}$vy processes}

As we saw in chapter 1 result \ref{wiener-poissonmrt}, martingales evolving on Wiener-Poisson spaces have the following representation formula:
\[ m_t = \int^t_0 \phi(t) dW_t + \int^t_0 \int_{\mathbb{R}_0} \psi(t,z) \tilde{N}(dt,dz), \]

where $ \phi(t), \psi(t,z)$ are predictable and $L^2$-integrable.
In this case, just as in section 2.2 and 2.3, the Clark-Ocone formula can be applied and gives an explicit form to $ \phi(t)$ and $\psi(t,z)$:

\begin{thm}[{\bf Clark-Ocone formula for l$\acute{e}$vy processes} \cite{Lokka}]
$\forall F \in L^2(\Omega) \cap \mathbb{D}_{1,2}$,
\[ F = \mathbb{E}(F) + \int^T_0 \mathbb{E}(D_t F \mid \mathcal{F}_t) dW_t +  \int^T_0 \int_{\mathbb{R}_0} \mathbb{E}(D_{t,z} F \mid \mathcal{F}_t) \tilde{N}(dt,dz). \]
\label{general-clark-ocone}
\end{thm}
Using theorem \ref{general-clark-ocone}, we can see that for all martingales $(m_t) \in \mathcal{M}^2 \cap \mathbb{D}_{1,2}$, 
\begin{equation}
m_t = \int^t_0 \mathbb{E}(D_s m_{\infty} \mid \mathcal{F}_s) dW_s +  \int^t_0 \int_{\mathbb{R}_0} \mathbb{E}(D_{s,z} m_{\infty} \mid \mathcal{F}_s) \tilde{N}(ds,dz).
\label{general-mrt}
\end{equation}
Here the concept of $\mathbb{D}_{1,2}$ is understood as such:
\[ \mathbb{D}_{1,2} = \big\{ F = \sum^{\infty}_n I_n(f_n) \mid \sum^{\infty}_n n \cdot n! \| f_n\|^2_n < \infty \big\} \subseteq L^2(\Omega), \]
where:
\[ I_n(f) = \sum_{\alpha} \int_{[0,T]^n} f(t)^{\alpha} dM^{t_1}_{\alpha_1} ... dM^{t_n}_{\alpha_n}, \]
$(M^{t_1}_{\alpha_1} ... M^{t_n}_{\alpha_n})$ is a martingale.
Based on this, we can re-write result \ref{general-mrt} as follow:
\[ m_t = \int^t_0 f(s) \cdot dM_s \]
where
\begin{eqnarray*}
f(t) = \left( \begin{array}{c}\mathbb{E}( D_t m_{\infty} \mid \mathcal{F}_t)\\ \mathbb{E}(  D_{t,z} m_{\infty} \mid \mathcal{F}_t) \end{array} \right) \mbox{, and } M_t = \left( \begin{array}{c} W_t \\ \tilde{N}(t,z) \end{array} \right)
\end{eqnarray*}
Hence we can re-write ($m_t$) as such:

\[ m_t =  \sum_{\alpha} \int^t_0 \mathbb{E}( D_{t,\alpha} m_{\infty} \mid \mathcal{F}_t) dM^{\alpha}_s \]

by setting $\alpha := (1,2)$ and :
\[ M^1_t = W_t \mbox{ , } D_{t,1} m_{\infty} = D_t m_{\infty}, \]
\[ M^2_t = \tilde{N}(t,z) \mbox{ , } D_{t,2} m_{\infty} = D_{t,z} m_{\infty}. \]

In this framework, in order to prove integrability, we give an alternative, more general definition to the operator $D_{t,\alpha}$:

\begin{defn} 
The function $D: \mathbb{D}_{1,2} \mapsto \bigoplus_{\alpha} L^2([0,T] \times \Omega, d\langle M^{\alpha} \rangle \times d\mathbb{P})$ is defined by:
\[ D_{t,\alpha} F := \sum^n_n nI_{n-1}(f^\alpha_n(\cdot, t)). \]
\label{Lokkalevyderiv}
Note that definition \ref{Lokkalevyderiv} is very similar to the definition of the Poisson-Malliavin derivative in result \ref{poissMallDer} and to the conventional definition of the Malliavin derivative for continuous processes set in result \ref{diffIn}. 
\end{defn}
Since $(m_t)$ is in $L^2(\mathcal{F}_T,\mathbb{P}) \cap \mathbb{D}_{1,2}$, it has the following decomposition:
\[ m_{\infty} = \sum_n I_n(f_n) \mbox{ , } \exists f_n \]
whenever $m_{\infty} < \infty$.
Then we can see that for any t inside [0,T] \cite{Lokka},

\begin{eqnarray*}
\mathbb{E}(m_t^2) &=& \mathbb{E}[  (\int^t_0 f(s) \cdot dM_s)^2] \\
&=& \mathbb{E}[ \sum_{\alpha} \int^t_0 (\mathbb{E}(D_{s,\alpha} m_{\infty}\mid \mathcal{F}_s))^2 d<M^{\alpha}>_s] \\
&=& \sum_{\alpha} \int^t_0 \mathbb{E}[ (\mathbb{E}(D_{s,\alpha} m_{\infty} \mid \mathcal{F}_s)^2] d<M^{\alpha}>_s \\
&=& \sum_{\alpha} \int^t_0 \mathbb{E}[ D_{s,\alpha} m_{\infty}^2] d<M^{\alpha}>_s \\
&=& \sum_{\alpha} \int^t_0 \| D_{s,\alpha} m_{\infty} \|^2_{L^2(\Omega)} d<M^{\alpha}>_s \\
&=& \sum_{\alpha} \int^t_0 \| \sum^{\infty}_n nI_{n-1}(f^{\alpha}_n( \cdot,s)) \|^2_{L^2(\Omega)} d<M^{\alpha}>_s  \mbox{ (definition \ref{Lokkalevyderiv}) }\\
&=& \sum^{\infty}_n n^2(n-1)! \sum_{\alpha} \int^t_0 \| f^{\alpha}_n( \cdot,s)) \|^2_{n-1} d<M^{\alpha}>_s  \\
&=& \sum^{\infty}_n n \cdot n! \|f_n\|^2_n < \infty.
\end{eqnarray*}

Hence the integral representation of result \ref{general-mrt} is well defined for any stochastic process $(m_t) \in \mathcal{M}^2 \cap \mathbb{D}_{1,2}$ where $m_{\infty}$ is finite.

\section{Generalization and Extension of the MRT and Clark-Ocone formula}

From the previous chapters we have seen that, on a diverse range of filtrations generated by different stochastic processes, all martingales in $\mathcal{M}^2$ can have a representation in the form:

\[ m_t = \int^t_0 \phi(s) \cdot d M^{\alpha}_s \]
for some remarkable martingale $M^\alpha$, and where $\phi(t)$ and $M^{\alpha}_t$ can be one of multi-dimensional. Additionally, for all martingales in a specific subset of $\mathcal{M}^2$, a.k.a $\mathcal{M}^2 \cap \mathbb{D}_{1,2}$ the intersection of square-integrable martingales and the Sobolev space $\mathbb{D}_{1,2}$, we have the Clark-Ocone formula specifying what the integrand $\phi(t)$ is:
 \[ m_t = \int^t_0 \mathbb{E}(D_s m_{\infty} \mid \mathcal{F}_s) \cdot d M^{\alpha}_s. \]
While this result is very useful and says a lot about martingale representation, it requires further investigation. First of all, we are interested in applying it to processes beyond the $\mathbb{D}_{1,2}$ space. As we will see, this Sobolev space can be too restrictive, especially when we are interested in changing elements of the stochastic base. Similarly, the MRT and Clark-Ocone formula need to remain under a change of measure on the probability space. From various applications, we know that changing measures and the use of the Girsanov theorem are of central importance to Financial Mathematics. But how do the MRT and the Clark-Ocone formula keep up with it? Furthermore, a modified version of these results happens to find an important application in the enlargement of filtrations. Recent literature has been exploring the eventuality of possessing extra information about the markets and the impact this has on trading strategies, notably in the context of insider trading. This requires the base filtration to be enlarged in order to take into account the new information available. How can we adapt the representations results we have to address this?\\
In the first section we introduce a new sobolev space $\mathbb{D}_{1,1}$, where $\mathbb{D}_{1,2} \subset \mathbb{D}_{1,1}$ \cite{Li}, and on which we can put forward various solutions to adapt the Clark-Ocone formula to the Girsanov theorem, which then becomes the Generalized Clark-Ocone representation theorem. This result is applicable to the B.M. as well as to general jump and L$\acute{e}$vy processes. In the latter section we then look at a Clark-Ocone type formula and measure-valued MRT that can be used in order to enlarge the base filtration.

\subsection{The MRT and Clark-Ocone formula beyond $\mathbb{D}_{1,2}$}

So far we have established  that martingales $(m_t)$ in $\mathcal{M}^2 \cap \mathbb{D}_{1,2}$ have an explicit integral representation as explained above. However, it often agreed that the Sobolev space $\mathbb{D}_{1,2}$ is too restrictive; As we will see in the next chapter, $\mathbb{D}_{1,1}$  lends itself better to changes of measures on the Brownian filtration. In this section we only consider the continuous case.

\begin{defn}
 Set \[ \mathcal{P} = \{ F(\omega) \mid F = f(W_{t_1},...,W_{t_n}), f \in \mathcal{C}^{\infty}(\mathbb{R}^{nd}), f -\mbox{ bounded} \}. \]
\end{defn}

\begin{thm}
$\forall F \in \mathcal{P}$, the Malliavin derivative $D_tF$ is equal to:
\[ D_t F = \sum^n_{i=1} \frac{\partial f}{\partial x_i} (W_{t_1},...,W_{t_n})1_{[0,t_i]}(t). \]
\label{BMpolynomDer}
\end{thm}

\begin{proof} As stated in proposition 2.14 of \cite{Li} and Theorem 1.6 of Section 2 , we can see that:
\begin{eqnarray*} 
D_{\gamma} F = \sum^n_{i=1} \frac{\partial f}{\partial x_i} (W_{t_1},...,W_{t_n}) \gamma(t_i) &=& \int^T_0 D_t F \dot \gamma dt \\
\sum^n_{i=1} \frac{\partial f}{\partial x_i} (W_{t_1},...,W_{t_n}) \gamma(t_i) &=& \sum^n_{i=1} \frac{\partial f}{\partial x_i} (W_{t_1},...,W_{t_n}) \int^{t_i}_0 \dot \gamma(s) ds \\
&=& \sum^n_{i=1} \frac{\partial f}{\partial x_i} (W_{t_1},...,W_{t_n}) \int^T_0 1_{[0,t_i]}\dot \gamma(s) ds \\
&=& \int^T_0 \sum^n_{i=1} \frac{\partial f}{\partial x_i} (W_{t_1},...,W_{t_n})  1_{[0,t_i]}\dot \gamma(s) ds \\
\mbox{So }  \int^T_0 D_t F \dot \gamma(t) dt &=& \int^T_0 \sum^n_{i=1} \frac{\partial f}{\partial x_i} (W_{t_1},...,W_{t_n})  1_{[0,t_i]}\dot \gamma(s) ds \\
\mbox{Hence } D_t F &=& \sum^n_{i=1} \frac{\partial f}{\partial x_i} (W_{t_1},...,W_{t_n})  1_{[0,t_i]}(t).
\end{eqnarray*}
\end{proof}

\begin{defn}[{\bf Sobolev Space $\mathbb{D}_{1,1}$}]
The banach space $\mathbb{D}_{1,1}$ is the closure of $\mathcal{P}$ under the following $L^2([0,T])$ norm:
\[ \| F\|_{1,1} = \mathbb{E}\big( |F| + \| D_t F \|_{L^2([0,T])} \big) \]
\label{D11}
\end{defn}

The gradient derivative DF=($D^1F,...,D^dF$) has components as stated by Theorem \ref{BMpolynomDer}:
\[ D^i_t F = \sum^n_{i=1} \frac{\partial}{\partial x^{ij}} f(W_{t_1},...,W_{t_n}) 1_{[0,t_i]}(t) \]

\begin{thm}[{\bf Clark-Ocone theorem in $\mathbb{D}_{1,1}$} \cite{Karatzas}]
$\forall F \in \mathbb{D}_{1,1}$, we have the following integral representation:
\[ F = \mathbb{E}(F) + \int^T_0 \mathbb{E}(D_t F \mid \mathcal{F}_t) dW_t. \]
\label{ClarkOconeD11}
\end{thm}

\begin{proof} for F $\in \mathbb{D}_{1,1}$, take $\{F_n\}_{n \in \mathbb{N}} \subseteq \mathcal{P}$ where $\lim_{n \rightarrow \infty} \| F_n-F\|_{1,1} =0$. $m(t) = \mathbb{E}(F \mid \mathcal{F}_t)$ and $m_n(t) = \mathbb{E}(F_n \mid \mathcal{F}_t)$ are martingales and hence, using result \ref{itorepth} from chapter 1,

\[ m(t) = \mathbb{E}(F) + \int^t_0 \phi(s)dW_s, \]
\[ m_n(t) = \mathbb{E}(F_n) + \int^t_0 \phi_n(s)dW_s, \]

\noindent where $\phi, \phi_n$ are square integrable. Since $m_n \in \mathcal{P} \subseteq \mathbb{D}_{1,2}$, we know from result \ref{clarkocone} in Chapter 2 that $\phi_n(t) = \mathbb{E}(D_tF_n \mid \mathcal{F}_t)$.  Also, for any $\epsilon >0$, 

\begin{eqnarray}
\mathbb{P}( \max_{0 \leq t \leq T}| m_n(t) - m(t)| > \epsilon ) &\leq& \frac{1}{\epsilon} \mathbb{E}|m_n(t) - m(t)| \mbox { (Doob's martingale inequality)}\nonumber \\
&=& \frac{1}{\epsilon} \mathbb{E}|\mathbb{E}(F_n \mid \mathcal{F}_t) - \mathbb{E}(F \mid \mathcal{F}_t)|  = \frac{1}{\epsilon} \mathbb{E}|\mathbb{E}(F_n  - F \mid \mathcal{F}_t)|  \nonumber \\
&=& \frac{1}{\epsilon} \mathbb{E}|F_n  - F | \rightarrow 0 \mbox{ as } n \rightarrow \infty. 
\label{Pmax}
\end{eqnarray}

 Since $ \mathbb{E}|F_n  - F | \leq \| F_n - f \|_{1,1} \rightarrow 0$.
Additionally, the Burkholder-Gundy inequality \cite{Rogers} shows that for any $\lambda > 0$ and $\delta \in (0,1)$,

\begin{align*}
\mathbb{P}( \langle m_n - m\rangle_T > 4\lambda^2 , \max_{0 \leq t \leq T} |m_n(t) - m(t)|  \leq \delta \lambda) &\leq \delta^2 \mathbb{P}(\langle m_n-m\rangle_T > \lambda^2) 
\end{align*}
As 
\begin{align*}
\langle m_n - m\rangle_T &=  \langle \mathbb{E}(Fn-F) + \int^T_0 \phi_n(s) - \phi(s) d W_s \rangle \\
&= \int^T_0 |\phi_n(s) - \phi(s)|^2 ds,
\end{align*}

\noindent we get:

\begin{equation}
 \mathbb{P}( \int^T_0 |\phi_n(s) - \phi(s)|^2 ds > 4\lambda^2 , \max_{0 \leq t \leq T} |m_n(t) - m(t)|  \leq \delta \lambda)  =  \delta^2 \mathbb{P}(\int^T_0 |\phi_n(s) - \phi(s)|^2 ds  > \lambda^2). \nonumber \\
\end{equation}
Additionally, we have: 
\begin{align}
\mathbb{P}(\int^T_0 |\phi_n-\phi|^2(s) ds > 4 \lambda^2) = \mathbb{P}( \int^T_0 |\phi_n(s) - \phi(s)|^2 ds > 4\lambda^2 &, \max_{0 \leq t \leq T} |m_n(t) - m(t)|  \leq \delta \lambda)  \nonumber \\  
+  \mathbb{P}( \int^T_0 |\phi_n(s) - \phi(s)|^2 ds > 4\lambda^2 &, \max_{0 \leq t \leq T} |m_n(t) - m(t)|  > \delta \lambda)   \nonumber \\
\leq \delta^2 \mathbb{P}(\int^T_0 |\phi_n(s) - \phi(s)|^2 ds  > \lambda^2&) \nonumber \\
+ \mathbb{P}( \int^T_0 |\phi_n(s) - \phi(s)|^2 ds > 4\lambda^2 &, \max_{0 \leq t \leq T} |m_n(t) - m(t)|  > \delta \lambda)  \nonumber \\
\leq \delta^2 + \mathbb{P}( \max_{0 \leq t \leq T} |m_n(t) - m(t)|  > &\delta \lambda). 
\label{eqnphi}
\end{align}

Results \ref{Pmax} and \ref{eqnphi} imply that

\[ \int^T_0 |\phi_n-\phi|^2(s) ds  \rightarrow 0 \mbox{ in $\mathbb{P}$ as } n \rightarrow \infty. \]

At the same time, we have:
\begin{eqnarray}
\mathbb{E}( \int^T_0 | \phi_n(s) - \mathbb{E}( D_s F \mid \mathcal{F}_s) | ds) &=& \mathbb{E}( \int^T_0 |  \mathbb{E}( D_s F_n \mid \mathcal{F}_s) - \mathbb{E}( D_s F \mid \mathcal{F}_s) | ds) \nonumber \\
&=& \mathbb{E}( \int^T_0 |  \mathbb{E}( D_s F_n - D_s F \mid \mathcal{F}_s) | ds) \nonumber \\
&\leq& \mathbb{E}( \int^T_0 |  D_s( F_n - F) | ds) \nonumber \\
&\leq& \sqrt{T} \mathbb{E}[ (\int^T_0 |  D_s( F_n - F) |^2 ds)^{1/2}] \mbox{ (Cauchy-Schwartz)} \nonumber \\
&\equiv& \sqrt{T} \mathbb{E}[ (\sum^d_{i=1} \| D^i (F_n-F) \|^2)^{1/2}] \nonumber \\
&\leq& \sqrt{T} \| F_n - F\|_{1,1} \rightarrow 0. 
\label{dtconv}
\end{eqnarray}

\ref{eqnphi} and \ref{dtconv} together show that $\mathbb{E}(D_tF \mid{F}_t) = \phi(t)$ $dt \times d\mathbb{P}$- a.s. .

\end{proof}

Hence, Theorem \ref{ClarkOconeD11} shows that a wider class of martingales can have an explicit integral representation: $\forall (m_t) \in \mathcal{M}^2 \cap \mathbb{D}_{1,1}$, we have:
\[ m_t = \int^t_0 \mathbb{E}(D_t m_{\infty} \mid \mathcal{F}_s) dW_s. \]


\subsection{The Generalized/Girsanov Clark-ocone formula}

\subsubsection{Generalized Clark representation formula for continuous martingales}

Here, we start with the setting developed in Ocone and Karatzas \cite{Ocone}: we are on a probablility space $(\Omega,\mathcal{F},\mathbb{P})$ generated by a $\mathbb{R}^d$-B.M. Denote

\[ \mathcal{F}_t = \sigma (W_s \mid 0 \leq s \leq t)\]
the $\mathbb{P}$-augmentation of the original filtration.\\

Set the following Radon-Nikodym derivative:

\[ \frac{d\mathbb{\tilde{P}}}{d\mathbb{P}} = Z_t = e^{-\int^t_0 \theta_s dW_s - 1/2 \int^t_0 \theta^2_s ds }, \]
where $\theta(t)$ is an $\mathbb{R}^d$-valued and $\mathcal{F}_t$-measurable process. We know from the conventional Girsanov theorem that 
\begin{equation}
\tilde{W}_t = W_t + \int^t_0 \theta_s ds
\label{driftBM}
\end{equation}
is a $\tilde{\mathbb{P}}$-B.M.

Here, we borrow the concepts of Malliavin Calculus developed in chapter 2 for the Malliavin derivative and the $\mathbb{D}_{1,2}$ norm:
\[ \mathcal{P} = \big\{ F \mid F(\omega) = f(W_{t_1},...,W_{t_n}), f: \mathbb{R}^{n\times d} \rightarrow \mathbb{R} \}. \]

We then set the gradient DF$(\omega)$ = ($D^1F,...,D^dF$) using theorem \ref{BMpolynomDer}
\[ D^iF = (D_tF)^i = \sum^n_{j=1} \frac{\partial f}{\partial x_{ij}} (W_{t_1},...,W_{t_n})1_{[0,t_j]}(t) \]
for any i $\leq d$.

The closure of $\mathcal{P}$ under the following norm
\[ \| |F| \|_{1,1} =\mathbb{E}(|F| + \| D_tF\|_{L^2([0,T])}) \]

then makes a Banach space $\mathbb{D}_{1,1}$ very similar but larger than the conventional Sobolev space $\mathbb{D}_{1,2}$.\\
We know from Theorem \ref{ClarkOconeD11} that $\forall F \in \mathbb{D}_{1,1}$, we have:
\[ F = \mathbb{E}(F) + \int^T_0 \mathbb{E}(D_t F \mid \mathcal{F}_t) dW_t. \]

Here we focus on $\mathbb{D}_{1,1}$ in order to avoid adding extra constraints to the theorem that will follow. Indeed, in this theorem we will want to give an integral representation to $\mathbb{\tilde{E}}(F \mid \mathcal{F}_t)$ using the Bayes formula, and apply the Clark-Ocone formula to F$Z_T$. In $\mathbb{D}_{1,1}$, it is the case that $F\in L^2(\mathbb{\tilde{P}}) \Longrightarrow FZ_T \in L^2(\mathbb{P})$, but not in $\mathbb{D}_{1,2}$ without restrictive moment constrains on F and DF.\cite{Ocone}

\begin{thm}[{\bf Generalized Clark representation formula} \cite{Ocone}]
$\forall F \in \mathbb{D}_{1,1}$ with bounded $\theta$ such as:
\[ \mathbb{E}( |F| Z_T) < \infty, \]
\[ \mathbb{E}( \|DF\| Z_T) < \infty, \]
\[ \mathbb{E}( |F| Z_T \| \int^T_0D\theta(s)dW_s+\int^T_0D\theta(s)\cdot \theta(s)ds\| ) < \infty, \]
then $FZ_T \in \mathbb{D}_{1,1}$ and the following holds:
\[ F = \tilde{\mathbb{E}}(F) + \int^T_0 \Big[ \tilde{\mathbb{E}}(D_t F \mid \mathcal{F}_t)- \tilde{\mathbb{E}}(F \int^T_t D_t \theta(u) d \tilde{W}_u \mid \mathcal{F}_t)\Big] d\tilde{W}_t. \]  
\label{generalizedClark}
\end{thm}

\begin{proof} For the purpose of this proof, note $Z_t = e^G$, where $G = -\int^t_0 \theta_s dW_s - 1/2 \int^t_0 \theta^2_s ds$
We can see from proposition 2.3 of \cite{Ocone}:

\begin{eqnarray*}
\| |\int^T_0 \theta_s dW_s| \|_{1,1} < \infty, \\
\| |\int^T_0 |\theta^n_s|^2 ds| \|_{1,1} < \infty. 
\end{eqnarray*}

Hence $G\in \mathbb{D}_{1,1}$. Additionally, $D \int^T_0 |\theta|^2ds = 2 \int^T_0 D \theta \cdot \theta ds $. So
\begin{eqnarray*}
\| |F Z_T| \|_{1,1} &=& \mathbb{E}( |F Z_T| + |D(FZ_T)|) \\
&=& \mathbb{E}( |F Z_T| ) + \mathbb{E}(Z_T \|DF\|+ |F| \|DZ_T\|) \\
&=& \mathbb{E}( |F Z_T| ) + \mathbb{E}(Z_T \|DF\|)+\mathbb{E}( |F| Z_T \| \int^T_0 D\theta dW_s + \int^T_0 D \theta \cdot \theta ds\|)  \\
&<& \infty \mbox{, as per conditions given in the theorem.}
\end{eqnarray*}
Hence $FZ_T \in \mathbb{D}_{1,1}$.\\
From above we know that:

\begin{eqnarray*}
D_t(FZ_T) &=& Z_T D_tF - FZ_TD_t (\int^T_0\theta dW_s + \int^T_0  \theta^2 ds) \\
&=& Z_T \big( D_tF - F(\int^T_t D_s\theta dW_s + \theta_t + D_t(\int^T_0 \theta^2 ds))\big) \\
&=& Z_T \big( D_tF - F(\int^T_t D_s\theta dW_s + \theta_t +\int^T_t D_s\theta \cdot \theta ds)\big) \\
&=&  Z_T \big( D_tF - F( \theta_t +\int^T_t D_s\theta ( dW_s + \theta ds))\big) \\
&=& Z_T \big( D_tF - F( \theta_t +\int^T_t D_s\theta d\tilde{W}_s)\big).
\end{eqnarray*}

We also know that $ \tilde{\mathbb{E}}(F \mid \mathcal{F}_t) = \frac{\mathbb{E}(FZ_T \mid \mathcal{F}_t)}{\mathbb{E}(Z_T \mid \mathcal{F}_t)}$, as a basic property of conditional expectations under change of measures. Since $Z_T$ is an $\mathbb{P}$-mg, $\mathbb{E}(Z_T \mid \mathcal{F}_t)= Z_t $. Here, note $\Lambda_t = 1/Z_t$. Then:

\begin{eqnarray*}
\tilde{\mathbb{E}}(F \mid \mathcal{F}_t) &=& \Lambda_t \mathbb{E}(FZ_T \mid \mathcal{F}_t) \\
&=& \Lambda_t \big( \mathbb{E}(FZ_T ) + \int^t_0 \mathbb{E}( D_t FZ_T\mid \mathcal{F}_s) dW_s \big) \mbox{ (result of chapter 2).} \\
\end{eqnarray*}

We recall that $d \Lambda_t = \Lambda_t \theta_t d \tilde{W}_t$. This gives:

\begin{eqnarray}
d \tilde{\mathbb{E}}(F \mid \mathcal{F}_t) &=& d [ \Lambda_t \big( \mathbb{E}(FZ_T ) + \int^t_0 \mathbb{E}( D_s FZ_T\mid \mathcal{F}_s) dW_s \big)] \nonumber \\
&=& d \Lambda_t \big( \mathbb{E}(FZ_T ) + \int^t_0 \mathbb{E}( D_s FZ_T\mid \mathcal{F}_s) dW_s \big) + \Lambda_t \big( d \mathbb{E}(FZ_T ) + d \int^t_0 \mathbb{E}( D_s FZ_T\mid \mathcal{F}_s) dW_s \big) \nonumber \\
&&+ \langle d \Lambda_t ,  d \mathbb{E}(FZ_T ) + d \int^t_0 \mathbb{E}( D_s FZ_T\mid \mathcal{F}_s) dW_s \rangle \nonumber \\
&=& \Lambda_t \theta_t \big( \mathbb{E}(FZ_T ) + \int^t_0 \mathbb{E}( D_s FZ_T\mid \mathcal{F}_s) dW_s \big) d\tilde{W}_t + \Lambda_t  \mathbb{E}( D_t FZ_T\mid \mathcal{F}_t) dW_t \nonumber \\
&&+ \mathbb{E}( D_t FZ_T\mid \mathcal{F}_t) \Lambda_t \theta_t  \langle d \tilde{W}_t ,  dW_s \rangle \nonumber \\
&=&  \theta_t \tilde{\mathbb{E}}(F \mid \mathcal{F}_t)   d\tilde{W}_t + \Lambda_t  \mathbb{E}( D_t FZ_T\mid \mathcal{F}_t) dW_t  + \mathbb{E}( D_t FZ_T\mid \mathcal{F}_t) \Lambda_t \theta_t dt \nonumber \\
&=& \big( \Lambda_t \mathbb{E}( D_t FZ_T\mid \mathcal{F}_t) + \theta_t \tilde{\mathbb{E}}(F \mid \mathcal{F}_t) \big) d\tilde{W}_t.
\label{intermediate2}
\end{eqnarray}

We also have:

\begin{eqnarray*}
\Lambda_t \mathbb{E}( D_t FZ_T\mid \mathcal{F}_t) &=& \Lambda_t \mathbb{E}(  Z_T \big( D_tF - F( \theta_t +\int^T_t D_t\theta d\tilde{W}_s)\mid \mathcal{F}_t) \\
&=& \Lambda_t \mathbb{E}(  Z_TD_tF\mid \mathcal{F}_t) - \theta_t \Lambda_t \mathbb{E}(  Z_TF\mid \mathcal{F}_t) - \Lambda_t \mathbb{E}(  Z_TF\int^T_t D_t\theta d\tilde{W}_s\mid \mathcal{F}_t) \\
&=& \tilde{\mathbb{E}}( D_tF\mid \mathcal{F}_t)-\theta_t \tilde{\mathbb{E}}(  F\mid \mathcal{F}_t)-\tilde{\mathbb{E}}(  F\int^T_t D_t\theta d\tilde{W}_s\mid \mathcal{F}_t). 
\end{eqnarray*}

Hence result \ref{intermediate2} turns into:
\begin{equation}
d \tilde{\mathbb{E}}(F \mid \mathcal{F}_t) = \big[ \tilde{\mathbb{E}}( D_tF\mid \mathcal{F}_t)- \tilde{\mathbb{E}}(  F\int^T_t D_t\theta d\tilde{W}_s\mid \mathcal{F}_t) \big] d\tilde{W}_t.
\label{intermediate3}
\end{equation}

Result \ref{intermediate3} clearly shows that $\tilde{\mathbb{E}}(F \mid \mathcal{F}_t)$ is a $\tilde{\mathbb{P}}$-mg. Therefore:
\[ F - \tilde{\mathbb{E}}(F)  = \int^T_0 \big[  \tilde{\mathbb{E}}( D_tF\mid \mathcal{F}_s)- \tilde{\mathbb{E}}(  F\int^T_s D_t\theta d\tilde{W}_u\mid \mathcal{F}_s) \big] d\tilde{W}_s. \]

\end{proof}

Hence, for all $\mathbb{P}$-martingales $(m_t) \in \mathcal{M}^2 \cap \mathbb{D}_{1,1}$, we can maintain an explicit representation even under various changes of measures: provided $m_{\infty} < \infty$ where $\mathbb{E}(m_{\infty} \mid \mathcal{F}_t) = m_t$, we have:
\begin{eqnarray*}
m_t &=& \mathbb{E}(m_{\infty} \mid \mathcal{F}_t) \\
&=& \mathbb{E}(  \tilde{\mathbb{E}}(m_{\infty})  + \int^T_0 \big[  \tilde{\mathbb{E}}( D_t m_{\infty}\mid \mathcal{F}_s)- \tilde{\mathbb{E}}(  m_{\infty} \int^T_s D_t\theta d\tilde{W}_u\mid \mathcal{F}_s) \big] d\tilde{W}_s \mid \mathcal{F}_t)  \\
&=&   \tilde{\mathbb{E}}(m_{\infty}) + \mathbb{E}(  \int^T_0 \big[  \tilde{\mathbb{E}}( D_t m_{\infty}\mid \mathcal{F}_s)- \tilde{\mathbb{E}}(  m_{\infty} \int^T_s D_t\theta d\tilde{W}_u\mid \mathcal{F}_s) \big] d\tilde{W}_s \mid \mathcal{F}_t)\\
&=&  \tilde{\mathbb{E}}(m_{\infty})+ \int^t_0 \big[  \tilde{\mathbb{E}}( D_t m_{\infty}\mid \mathcal{F}_s)- \tilde{\mathbb{E}}(  m_{\infty} \int^T_s D_t\theta d\tilde{W}_u\mid \mathcal{F}_s) \big] dW_s  \\
&&+ \int^T_0 \big[\tilde{\mathbb{E}}( D_t m_{\infty}\mid \mathcal{F}_s)- \tilde{\mathbb{E}}(  m_{\infty} \int^T_s D_t\theta d\tilde{W}_u\mid \mathcal{F}_s) \big] \theta_s ds \\
&=& \tilde{\mathbb{E}}(m_{\infty})+ \int^t_0 \big[  \tilde{\mathbb{E}}( D_t m_{\infty}\mid \mathcal{F}_s)- \tilde{\mathbb{E}}(  m_{\infty} \int^T_s D_t\theta d\tilde{W}_u\mid \mathcal{F}_s) \big] d\tilde{W}_s \\
&&+\int^T_t \big[\tilde{\mathbb{E}}( D_t m_{\infty}\mid \mathcal{F}_s)- \tilde{\mathbb{E}}(  m_{\infty} \int^T_s D_t\theta d\tilde{W}_u\mid \mathcal{F}_s) \big] \theta_s ds.
\end{eqnarray*}

 Note: We cannot directly apply the classical Clark-Ocone formula developed in Chapter 2 to obtain an integral representation with respect to $\tilde{W}_t$, as $m_t \in \mathcal{M}^2 \cap \mathbb{D}_{1,1}$ is not necessarily $\tilde{\mathcal{F}}_t$-adapted. $\tilde{\mathcal{F}}_t$ is generated by $\tilde{W}_t$ for any t in [0,T], and very often it is the case that $\tilde{\mathcal{F}}_T \subset \mathcal{F}_T$. \cite{Oksendal2}

\subsubsection{Generalized Clark representation formula for L$\acute{e}$vy and pure jump processes}

The result investigated in section 1.1 can be extended beyond continuous martingales and the B.M. driving them. Indeed, It is possible to prove that it is also applicable to general L$\acute{e}$vy processes of the type explored in Chapter 1 and 2.\\
We work on the filtration generated by a L$\acute{e}$vy process  $L_t$ such as, $\forall (m_t) \in \mathcal{M}^2 \cap \mathbb{D}_{1,2}$, we have:
\[ m_t = \int^t_0 \mathbb{E}(D_s m_{\infty} \mid \mathcal{F}_s) dW_s +  \int^t_0 \int_{\mathbb{R}_0} \mathbb{E}(D_{s,z} m_{\infty} \mid \mathcal{F}_s) \tilde{N}(ds,dz).\]

Again, we perform a change of measure as in Nunno et al. \cite{Oksendal3}:
\begin{eqnarray*}
\frac{d \mathbb{Q}}{d \mathbb{P}} = Z_t &=&\exp \Big\{ - \int^t_0 u_s dW_s - \int^t_0 u^2_sds \\
&& +\int^T_0 \int_{\mathbb{R}_0} \log(1-\theta(s,x)) + \theta(s,x) \upsilon(dx) ds \\
&& +\int^T_0 \int_{\mathbb{R}_0} \log(1-\theta(s,x)) \tilde{N}(ds,dx) \Big\}
\end{eqnarray*}
where $\theta(s,x) \leq 1$ for s$\in [0,T]$ , x $\in \mathbb{R}_0$ and $u_s$ is a $\mathbb{F}$-predictable process.

\begin{thm}[{\bf Generalized Clark-Ocone theorem for L$\acute{e}$vy processes} \cite{Oksendal3}]
 For all F $\in L^2(\mathbb{P}) \cap  L^2(\mathbb{Q})$ where F is $\mathcal{F}_T$-measurable and where:
\begin{itemize}
\item $\theta \in L^2(\mathbb{P} \times \lambda \times \upsilon)$,
\item $D_{t,x} \theta$ is Skorohod integrable,
\end{itemize}
the below representation holds

\begin{eqnarray*}
 F &=& \mathbb{E}_{\mathbb{Q}}(F) + \int^T_0 \mathbb{E}_{\mathbb{Q}}(D_t F - F \int^T_t D_t u_s dW^Q_s \mid \mathcal{F}_t) d W^Q_t \\
&& +\int^T_0 \int_{\mathbb{R}_0}\mathbb{E}_{\mathbb{Q}}(F (\tilde{H}-1)+\tilde{H}D_{t,x}F \mid \mathcal{F}_t) d \tilde{N}^Q(dt,dx).
\end{eqnarray*}

Here:
\begin{eqnarray*}
\tilde{H} &=& \exp \Big\{ \int^t_0 \int_{\mathbb{R}_0} [ D_{t,x} \theta(s,z) + \log(1-\frac{D_{t,x} \theta(s,z)}{1-\theta(s,z)})(1-\theta(s,z))] \upsilon(dz)ds \\
&& + \log(1-\frac{D_{t,x} \theta(s,z)}{1-\theta(s,z)})\tilde{N}_Q(ds,dz) \Big\}, \\
\tilde{N}_Q(ds,dz)  &=& \theta(t,x)\upsilon(dx)dt + \tilde{N}(ds,dz), \\
d W^Q_t &=& u_tdt + dW_t. \\ 
\end{eqnarray*}
\label{generalizedClarkLevy}
\end{thm}

As in section 1.1 of this chapter, we can apply \ref{generalizedClarkLevy} to martingales: If $m_{\infty} < \infty$ and $m_{\infty} \in L^2(\mathbb{P}) \cap L^2(\mathbb{Q})$, then:

\begin{eqnarray*}
m_{\infty} &=&  \mathbb{E}_{\mathbb{Q}}(m_{\infty}) + \int^T_0 \mathbb{E}_{\mathbb{Q}}(D_t m_{\infty} - m_{\infty} \int^T_t D_t u_s dW^Q_s \mid \mathcal{F}_t) d W^Q_t \\
&& +\int^T_0 \int_{\mathbb{R}_0}\mathbb{E}_{\mathbb{Q}}(m_{\infty} (\tilde{H}-1)+\tilde{H}D_{t,x}m_{\infty} \mid \mathcal{F}_t) d \tilde{N}^Q(dt,dx), \\
m_t &=& \mathbb{E}_{\mathbb{P}}( m_{\infty} \mid \mathcal{F}_t) \\
&=& \mathbb{E}_{\mathbb{P}}(  \mathbb{E}_{\mathbb{Q}}(m_{\infty}) + \int^T_0 \mathbb{E}_{\mathbb{Q}}(D_t m_{\infty} - m_{\infty} \int^T_t D_t u_s dW^Q_s \mid \mathcal{F}_t) d W^Q_t \\
&& +\int^T_0 \int_{\mathbb{R}_0}\mathbb{E}_{\mathbb{Q}}(m_{\infty} (\tilde{H}-1)+\tilde{H}D_{t,x}m_{\infty} \mid \mathcal{F}_t) d \tilde{N}^Q(dt,dx) \mid \mathcal{F}_t) \\
&=& \mathbb{E}_{\mathbb{Q}}(m_{\infty}) +  \int^t_0 \mathbb{E}_{\mathbb{Q}}(D_t m_{\infty} - m_{\infty} \int^T_s D_t u_a dW^Q_a \mid \mathcal{F}_s) d W^Q_s  \\
&&+ \int^T_t \mathbb{E}_{\mathbb{Q}}(D_t m_{\infty} - m_{\infty} \int^T_s D_t u_a dW^Q_a \mid \mathcal{F}_s) u_s d s\\
&&+ \int^t_0 \int_{\mathbb{R}_0}\mathbb{E}_{\mathbb{Q}}(m_{\infty} (\tilde{H}-1)+\tilde{H}D_{t,x}m_{\infty} \mid \mathcal{F}_s) d \tilde{N}^Q(ds,dx) \\
&&+ \int^T_t \int_{\mathbb{R}_0}\mathbb{E}_{\mathbb{Q}}(m_{\infty} (\tilde{H}-1)+\tilde{H}D_{t,x}m_{\infty} \mid \mathcal{F}_s) d \theta(s,x) \upsilon(dx).
\end{eqnarray*}

\subsection{Enlargement of filtration: Poisson filtrations and insider trading (Wright et al. \cite{Wright})}

So far we have reviewed various ways to adapt the MRT and the Clark-Ocone formula under a change of probability measure. But the requirements of Financial Mathematics also lead towards the eventuality of changing the base filtration $\mathcal{F}$ which we work on. Doing so impacts on the representation of the stochastic processes driving the probability space, and hence the MRT. \\
We start with a conventional probability space $(\Omega, \mathcal{F}, \mathbb{P})$ where $\{\mathcal{F}_t\}_{t \leq 1}$ represents the regular information flow generated by a B.M. $\in \mathbb{R}^d$. Here, the time frame is t$ \in [0,1]$.Then comes a $\mathcal{F}_1$-measurable r.v. L that carries extra information. The agent who possesses that extra info L is otherwise known as an "insider" \cite{Wright} and their knowledge is represented by the enlarged filtration:
\[ \mathcal{G}_t = \mathcal{F}_t \vee \sigma(L) \mbox{, } \forall t \in [0,1]. \]

The $\mathcal{G}_t$-B.M. $\tilde{W}_t$ can be represented as such:
\[ W_t = \tilde{W}_t + \int^t_0 \mu^L_sds. \]

There is such a drift $\mu^L_s$ when what is known as "Jacod's Condition" \cite{Jacod} is satisfied:
\[ \mbox{ The regular conditional distributions of L given } \mathcal{F}_t \]
\[ \mbox{ are absolutely continuous with respect to the law of L } \forall t \in [0,1) \]
Then we can re-write the MRT as such:

\[ \forall m \in \mathcal{M}^2 \cap \mathbb{D}_{1,1}, m_t = \int^t_0 \mathbb{E}(D_s m_{\infty} \mid \mathcal{F}_s) d\tilde{W}_t + \mu^L_s ds \mbox{ , } \forall t \in [0,1]. \]

But how do we find an explicit form for $ \mu^L_s$? As developed in Wright et al.\cite{Wright}, by identifying the integral representation of $\mathbb{P}( L \in dx \mid \mathcal{F}_t)$ explicitly using the Clark-Ocone formula, we can get an expression for the drift $\mu^L_s$.
However, it is clear that $\mathbb{P}( L \in dx \mid \mathcal{F}_t)$  is not a simple random variable but a measure-valued random
variable. Therefore we need to re-create a measure-valued Poisson Malliavin calculus with its own Clark-Ocone-type formula  in order to make sense of an integral representation of $\mathbb{P}( L \in dx \mid \mathcal{F}_t)$. \cite{Wright}

\subsubsection{Use of Poisson-malliavin calculus and Clark-Ocone formula}

We work on a Poisson space $(\mathcal{B},\mathcal{F},\mathcal{P})$ defined as follow:

\begin{defn}[{\bf Poisson space}]
A poisson space is a triple $(\mathcal{B},\mathcal{F},\mathcal{P})$ where:
\begin{itemize}
\item $\mathcal{B}$ is a sequence space,
\item $\mathcal{P}$ is a probability measure under which $\tau_k : \mathcal{B} \rightarrow \mathbb{R}$ forms a sequence of i.i.d exponentially distributed,
\item $\mathcal{F}$ is the $\sigma$-field generated by  $\mathcal{B}$
\end{itemize}
\end{defn}

The n-th jump time of the Poisson process, $T_n$ , is then derived by:
\[ T_n = \sum^{n-1}_{k=0} \tau_k, \]

and the Poisson process itself, $N_t$, by:
\[ N_t(\omega) = \sum^{\infty}_{k=1} 1_{[T_k(\omega),\infty)}(t)\]
for any t in $\mathbb{R}$.
$(\mathcal{F}_t)_{t \geq 0}$ is the filtration generated by $(N_t)_{t \geq 0}$.

Define the set $\mathcal{S}$:
\[ \mathcal{S} = \{ F = f(T_1,...,T_n) \mid f \in C^{\infty}(\mathbb{R}^n) \forall n \geq 1\}, \]

and the closable linear operator $D^{\mathbb{R}}: L^2(\mathcal{B}) \rightarrow L^2(\mathcal{B} \times \mathbb{R}_+)$ for all F$\in \mathcal{S}$:
\[ D^{\mathbb{R}}_t F = - \sum^n_{k=1} 1_{[0,T_k]}(t) \partial_kf(T_1,...,T_n) \]

whenever t is in $\mathbb{R}_+$, the set of positive real numbers. We then extend $\mathcal{S}$ in Dom D $\subseteq L^2(\mathbb{B})$ with respect to:
\[ \|F\|_{L^2(\mathcal{B})} + \|DF\|_{L^2(\mathcal{B} \times \mathbb{R}_+)} \mbox{ , } F \in \mathcal{S}. \]

 We now introduce an important isomorphism $\Phi$ as:
\[ \Phi : \mathbb{M} \rightarrow \mathbb{R}^{\mathbb{N}}, \]
\[ \Phi(\mu) = ( \langle \mu, f_i \rangle)_{i \in \mathbb{N}} = ( \int_{\mathbb{R}} f_i d\mu)_{i \in \mathbb{N}}.  \]

Here, $\mathbb{M}$ denotes a space of measures:
\[ \mathbb{M} = \{ \mu \mid \mu: \mbox{ signed measure on } (\mathbb{R},\mathcal{B}) \}. \]

We define  $\mathcal{S}(\mathbb{M})$ and $ D^{\mathbb{M}}: \mathcal{S}(\mathbb{M}) \rightarrow  L^2(\mathcal{B} \times \mathbb{R}_+, \mathbb{M})$ in a very similar way to above:
\[ \mathcal{S}(\mathbb{M}) = \{ F = g(T_1,...,T_n,x)dx \mid g \in C^{\infty}(\mathbb{R}^{n+1}) \forall n \geq 1\}, \]
\[ D^{\mathbb{M}}_t F = - \sum^n_{k=1} 1_{[0,T_k]}(t) \partial_kg(T_1,...,T_n,x)dx. \]

As we have done in chapter 2, we can introduce a norm on $\mathcal{S}(\mathbb{M})$:
\[ \|F\|^{\mathbb{M}}_{1,2} = \mathbb{E}(|F|^2)^{1/2} + \mathbb{E}(\| |D^{\mathbb{M}}F| \|^2_2)^{1/2}, \]
and set $\mathbb{D}_{1,2}(\mathbb{M})$ as the closure of $\mathcal{S}(\mathbb{M})$ with respect to $\|\cdot\|^{\mathbb{M}}_{1,2}$.

\begin{thm}[Proposition 1 \cite{Wright}]
$ \forall F \in \mathbb{D}_{1,2}(\mathbb{M})$ and $f \in C_b(\mathbb{R})$, we have:
\begin{itemize}
\item $\langle F, f\rangle \in \mathbb{D}_{1,2}(\mathbb{M})$, and
\item $\langle D^{\mathbb{M}}_t F,f\rangle = D^{\mathbb{R}}_t \langle F,f\rangle$.
\end{itemize}
Additionally - Proposition 2 \cite{Wright} - for $F \in \mathbb{D}_{1,2}(\mathbb{M})$, we have:
\[ D^{\mathbb{M}} F = \Phi^{-1}(( D^{\mathbb{R}} \langle F,f_i \rangle)_{i \in \mathbb{N}}). \]
\label{prop1}
\end{thm}

\begin{thm}[Proposition 3 \cite{Wright}]
For $F_t$ adapted such as:
\[ \sup_{\| f \| \leq 1, f \in C_b(\mathbb{R})} \mathbb{E}\big[ \int^{\infty}_0 \langle F_t,f\rangle^2 dt \big] < \infty, \]
we have:
\[ \langle \int^{.}_0 F_t d\tilde{N}_t,f \rangle =  \int^{.}_0 \langle  F_t, f \rangle d\tilde{N}_t. \]
\label{prop3}
\end{thm}
 
Additionally to the results established above, we need to introduce a conditional expectation formula. Whenever F $\in \mathbb{M}$ is $\mathcal{F}$-measurable, $\langle F, f_i \rangle$ is also $\mathcal{F}$-measurable for $i \in \mathbb{N}$ and any f$ \in C_b(\mathbb{R})$. Denote $\mathcal{G} \subseteq \mathcal{F}$. Then $\forall i$, $\mathbb{E}[ \langle F, f_i \rangle \mid \mathcal{G}]$ is well-defined and we set:

\[ \mathbb{E}(F \mid \mathcal{G}) = \Phi^{-1}(\mathbb{E}[ \langle F, f_i \rangle \mid \mathcal{G}]_{i \in \mathbb{N}}). \]
So 
\[ \langle \mathbb{E}(F \mid \mathcal{G}) , f_i \rangle = \mathbb{E}( \langle F, f_i \rangle \mid \mathcal{G}) \mbox{ } \forall i.\]

Provided $ \|F\|_1 = \sup_{\| f \| \leq 1, f \in C_b(\mathbb{R})}\mathbb{E}\big[ |\langle F,f\rangle| \big] < \infty $, we have that $ | \mathbb{E} ( | \mathbb{E}( \langle F, f_i \rangle \mid \mathcal{G}) - \mathbb{E}( \langle F, f \rangle \mid \mathcal{G}) | ) | \leq \| f - f_i \| \| F \|_1  \rightarrow 0$ as $i \rightarrow \infty$. Therefore we can make the following assertion:

\begin{equation}
\mathbb{E}[ \langle F, f \rangle \mid \mathcal{G}) ] = \langle \mathbb{E}(F \mid \mathcal{G}) , f \rangle \mbox{ , } \forall f \in C_b(\mathbb{R}).
\label{wrightcondexp}
\end{equation}

At this stage, with the use of the results, definitions and theorems introduced so far in this section, we can prove an alternative Clark-Ocone formula for signed measures $F \in \mathbb{D}_{1,2}(\mathbb{M})$ that will then prove to be very useful in defining a representation for $ \mu^L_t$. \\

\begin{thm}[{\bf Clark-Ocone type formula for $F \in \mathbb{D}_{1,2}(\mathbb{M})$ } \cite{Wright}]
 For $F \in \mathbb{D}_{1,2}(\mathbb{M})$ , satisfying the boundedness condition set in Theorem \ref{prop3} and:
\[  \|F\|_1 = \sup_{\| f \| \leq 1, f \in C_b(\mathbb{R})}\mathbb{E}\big[ |\langle F,f\rangle| \big] < \infty, \]
F has a representation given by:
\[ F = \mathbb{E}(F) + \int^1_0 \mathbb{E}( D^{\mathbb{M}}_t F \mid \mathcal{F}_t) d \tilde{N}_t. \]
\label{wrighttheor1}
\end{thm}

\begin{proof}
To prove Theorem \ref{wrighttheor1}, we start with the result of Proposition 2 of \cite{Mensi}: for $F \in Dom D^{\mathbb{R}}$, 
\[ F = \mathbb{E}(F) + \int^1_0 \mathbb{E}( D^{\mathbb{R}}_t F \mid \mathcal{F}_t) d \tilde{N}_t. \]
From Theorem \ref{prop1}, we know that F $\in \mathbb{D}_{1,2}$ implies $\langle F,f\rangle \in \mathbb{D}_{1,2}$. Therefore,

\[ \langle F, f_i \rangle = \mathbb{E}(\langle F, f_i \rangle) + \int^1_0 \mathbb{E}(D^{\mathbb{R}}_t  \langle F, f_i \rangle \mid \mathcal{F}_t) d \tilde{N}_t. \]

On the side, we note that:

\begin{eqnarray}
\Phi^{-1}(( \mathbb{E}(\langle F, f_i \rangle))_{i \in \mathbb{N}}) &=& \Phi^{-1}(( \mathbb{E}(\langle F, f_i \rangle \mid \mathcal{F}_0))_{i \in \mathbb{N}}) \nonumber \\
&=& \Phi^{-1}(( \langle \mathbb{E}( F \mid \mathcal{F}_0), f_i \rangle)_{i \in \mathbb{N}}) \mbox{ , from result \ref{wrightcondexp}} \nonumber \\
&=& \mathbb{E}( F \mid \mathcal{F}_0) \nonumber \\
&=& \mathbb{E}( F),
\label{eqnEFwright}
\end{eqnarray}
 and that:

\begin{eqnarray}
\mathbb{E}(D^{\mathbb{R}}_t  \langle F, f_i \rangle \mid \mathcal{F}_t) &=& \mathbb{E}(  \langle D^{\mathbb{M}}_t F, f_i \rangle \mid \mathcal{F}_t) \mbox{ ,from Theorem \ref{prop1}} \nonumber \\
&=& \langle \mathbb{E}(  D^{\mathbb{M}}_t F \mid \mathcal{F}_t) , f_i \rangle. 
\label{eqneEDtwright}
\end{eqnarray}

Using results \ref{eqnEFwright} and \ref{eqneEDtwright}, we thus have:

\begin{eqnarray}
\langle F, f_i \rangle &=& \langle \mathbb{E}( F), f_i \rangle + \int^1_0 \langle \mathbb{E}(  D^{\mathbb{M}}_t F \mid \mathcal{F}_t) , f_i \rangle  d \tilde{N}_t \nonumber \\
&=& \langle \mathbb{E}( F), f_i \rangle +  \langle  \int^1_0 \mathbb{E}(  D^{\mathbb{M}}_t F \mid \mathcal{F}_t)   d \tilde{N}_t , f_i \rangle \mbox{ ,from Theorem \ref{prop3}} \nonumber \\
&=& \langle \mathbb{E}( F) +  \int^1_0 \mathbb{E}(  D^{\mathbb{M}}_t F \mid \mathcal{F}_t)   d \tilde{N}_t , f_i \rangle, \nonumber
\end{eqnarray}

and

\begin{eqnarray}
F = \Phi^{-1}( (\langle F, f_i \rangle )_{i \in \mathbb{N}}) &=& \Phi^{-1}( (\langle \mathbb{E}( F) +  \int^1_0 \mathbb{E}(  D^{\mathbb{M}}_t F \mid \mathcal{F}_t)   d \tilde{N}_t , f_i \rangle)_{i \in \mathbb{N}}) \nonumber \\
&=& \mathbb{E}( F) +  \int^1_0 \mathbb{E}(  D^{\mathbb{M}}_t F \mid \mathcal{F}_t)   d \tilde{N}_t. \nonumber
\end{eqnarray}

\end{proof}

\subsubsection{Information drift $\mu^L_t$}

We begin by re-writing the conditional expectations of L as:
\[ P_t( \cdot , dx) = \mathbb{P}( L \in dx \mid \mathcal{F}_t). \]

Here, we aim to apply Theorem \ref{wrighttheor1} to $ P_t( \cdot , dx)$ to obtain sufficient conditions for the existence of $\mu^L_t$ and derive a formula for it. If $P_t( \cdot , dx)$ satisfies the conditions of Theorem \ref{wrighttheor1}, then:

\begin{align*}
P_t( \cdot , dx) &= \mathbb{E}(  P_t( \cdot , dx)) + \int^1_0 \mathbb{E}( D^{\mathbb{M}}_sP_t( \cdot , dx) \mid \mathcal{F}_s) d\tilde{N}_s \\
\mbox{As: } \mathbb{E}(  P_t( \cdot , dx)) = \mathbb{E}(  \mathbb{E}(1_{\{L \in dx\}} \mid \mathcal{F}_t) \mid \mathcal{F}_0) &= \mathbb{E}( 1_{\{L \in dx\}}  \mid \mathcal{F}_0) =  P_0( \cdot , dx), \\
\mbox{ We have: } P_t( \cdot , dx) &= P_0( \cdot , dx)+ \int^1_0 \mathbb{E}( D^{\mathbb{M}}_sP_t( \cdot , dx) \mid \mathcal{F}_s) d\tilde{N}_s.
\end{align*}

Define the following set:
\[ \mathcal{V} = \{ \upsilon \in L^2(\mathcal{B} \times \mathbb{R}_+) \mid \upsilon(t) = f(t,T_1,...,T_n), f \in C^{\infty}_b(\mathbb{R}^{n+1}) \}. \]

$\mathcal{V}$ also has the property that $\forall \upsilon \in \mathcal{V}$, $\upsilon(y) = f(y,x_1,...,x_n) = 0$ whenever $y > x_n$.

\begin{thm}[Remark 2 \cite{Privault}]
 $\mathcal{V}$ is dense in $L^2(\mathcal{B} \times \mathbb{R}_+)$.
\label{denseV}
\end{thm}

Based on Theorem \ref{denseV}, Wright et al. \cite{Wright} prove the result below:

\begin{thm}[Corollary 1 \cite{Wright}]
For $u(\cdot,t) \in \mathbb{D}_{1,2}(\mathbb{M})$ such as  $u(\cdot,t)$ is $\mathcal{F}_t$-adapted $\forall t \in [0,1]$, $D^{\mathbb{M}}_s u_t = 0$ a.s. when $s > t$.
\label{collor1wright}
\end{thm}

As $P_t(\cdot,dx)$ is $\mathcal{F}_t$-adapted, Theorem \ref{collor1wright} applies and hence:

\[ D^{\mathbb{M}}_s P_t(\cdot,dx) = 0 \mbox{, as } s \geq t. \]
The above result justifies the following assertion, which is an extension to Proposition 3 in \cite{Mensi}:

\begin{thm}[{\bf Condition of existence and formula for $\mu^L_t$} \cite{Wright} ]
Given L and its conditional law  $P_t(\cdot,dx)$, if  $P_t(\cdot,dx)$ satisfies the conditions of Theorem \ref{wrighttheor1}, $P_t(\cdot,dx)$ is represented as:
\[ P_t(\cdot,dx) =  P_0( \cdot , dx)+ \int^1_0 \mathbb{E}( D^{\mathbb{M}}_sP_t( \cdot , dx) \mid \mathcal{F}_s) d\tilde{N}_s.\]

If $\exists g : \mathbb{R}_+ \times \mathcal{B} \times \mathbb{R}_+ \rightarrow \mathbb{R}$ measurable and a Stopping time S such as
\[ 1_{\{s \leq S\}}\mathbb{E}(D^{\mathbb{M}}_sP_t( \cdot , dx) \mid \mathcal{F}_s) = 1_{\{s \leq S\}}g_s(\cdot,x)P_s(\cdot,dx), \]
then
\[ \tilde{N}_t - \int^t_0 g_s(\cdot,L)ds \]
is a $\mathcal{G}_t$-martingale. In other words, $\mu^L_s =  g_s(\cdot,L)$.
\label{prop6wright}
\end{thm}

Applying Theorem \ref{prop6wright} to our initial problem, we get:\\
$\forall m \in \mathcal{M}^2 \cap \mathbb{D}_{1,1}$, we have the following integral representation:

\[ m_t = \int^t_0 \mathbb{E}(D_s m_{\infty} \mid \mathcal{F}_s) (d\tilde{W}_t + g_s(\cdot,L) ds) \mbox{ , } \forall t \in [0,1] \]
where
\[ \exists S \mbox{ - stopping time - s.a. } 1_{\{s \leq S\}}\mathbb{E}(D^{\mathbb{M}}_sP_t( \cdot , dx) \mid \mathcal{F}_s) = 1_{\{s \leq S\}}g_s(\cdot,x)P_s(\cdot,dx). \]

\subsection{Further Consideration}

We have reviewed a couple of occasions where the MRT and Clark-Ocone formula need to and can be generalized. However more can be done in order to make the MRT and Clark-Ocone formula applicable to larger sets of martingales and more situations. \\
\\ It is agreed that neither the $\mathbb{D}_{1,2}$ nor the $\mathbb{D}_{1,1}$ spaces are general enough for financial applications. Ideally, we want to be able to apply the concept of the Clark-Ocone representation formula to every $\mathcal{F}_T$-measurable F of $L^2(\Omega, \mathcal{F}, \mathbb{P})$ and its associated martingale space $\mathcal{M}^2$, for stochastic bases $(\Omega, \mathcal{F}, \mathbb{P})$ generated by any class of processes. One way of doing so is developed in Aase et al. \cite{Aase}, where a white-noise approach to Malliavin calculus is used in order to prove the following:

\[ \forall F \in \mathcal{G}^* \supset L^2(\mu) \mbox{ , } F = \mathbb{E}(F) + \int^T_0 \mathbb{E}(D_t F \mid \mathcal{F}_t) \diamond W_t dt, \]

\noindent where $D_t F = D_t F(\omega) = \frac{dF}{d\omega}$ is the generalized Malliavin derivative, $\diamond$ represents the Wick product and $W_t$ can be a scalar or multi-dimentional Gaussian, Poisson or combined Gaussian-Poisson white noise. The above formula holds on $\mathcal{G}^*$, which is a space of stochastic distributions. Additionally, $\mu$ represents the white-noise probability measure, hence $\mathcal{G}^* \supset L^2(\mu) $.  Another paper by Ustunel \cite{Ustunel} also offers a similar generalization of the Clark-Ocone formula for all F in $\mathbb{D}_{-\infty}$, which is the space of Meyer-Watanabe distributions. However, $\mathbb{D}_{-\infty} \subset \mathcal{G}^*$ and $\mathbb{D}_{-\infty} \neq \mathcal{G}^*$  \cite{Aase}\\
\\
Another way to approach integral martingale representation is through non-anticipative functionals as done in Cont \cite{Cont}. Instead of using regular Malliavin calculus, we employ the concepts of horizontal and vertical derivative $ \mathcal{D}_tF$ and $\bigtriangledown_x F$ for a  non-anticipative functional F = $(F_t)_{[0,T)}$. On a probability space $(\Omega, \mathcal{F}, (\mathcal{F}_t), \mathbb{P})$ evolves a continuous $\mathbb{R}^d$-valued semi-martingale X that generates the sigma-fields $\mathcal{F}^X_t$. Then for every $\mathcal{F}^X_t$-adapted Y it can be shown that
\[ Y_t = F_t(X_t,A_t) \]
where  $\langle X\rangle_t =  \int^t_0 A_udu$ and F is a functional representing the dependence of Y on X and its quadratic variation. Based on this setting, it is possible to show an alternative form of the martingale representation theorem:
\[ \forall Y \in L^2(X)\cap\mathcal{M}^2 \mbox{ s.a. Y is } \mathcal{F}^X_t \mbox{-adapted, } Y_T = Y_0 +\int^T_0 \bigtriangledown_xYdX. \]
This result is of particular interest as it is computationally less intensive than the regular Malliavin derivative. $\bigtriangledown_xY$ can be calculated pathwise, and hence lends itself better to numerical computations. Note that when X=W is a Brownian motion, the vertical derivative $\bigtriangledown_W$ can be related to the Malliavin derivative.\\
\\
This paper has covered cases evolving on the standard space, a.k.a $\mathbb{R}$. However aspects of nonstandard analysis and applications of nonstandard stochastics to finance are of increased interest, as the hyperfinite versions of the regular option pricing models are better at outlining the connections between discrete and continuous trading models \cite{Cutland}. 
To apply pricing and trading models to $^*\mathbb{R} \setminus\mathbb{R}$, we would be looking at creating a version of the MRT and of the Clark-Ocone formula beyond the standard space. There is currently no rigorous  nonstandard proof of the MRT in the form we have reviewed in this paper, since the filtrations of the nonstandard setting are much too rich and would result in integrals that are not well-defined \cite{Lindstrom}. One possible area of expansion in this direction would be to follow the idea of Lindstr$\o$m \cite{Lindstrom2}. There, the Brownian motion in equation \ref{itorepth} is replaced by an Anderson's random walk, which is a binomial random walk with an infinitesimal increment $\delta$ such as $\delta = T/N$, $N \in$  $^*\mathbb{R} \setminus\mathbb{R}$ on the interval [0,T]  \cite{Anderson}. It is sometimes the case that hyperfinite stochastic processes have similar properties to standard ones, but this does not always hold, hence the need for further research in this area.

\end{document}